\documentclass[11pt,reqno]{amsart}
\usepackage{enumerate}
\usepackage{amsrefs}
\usepackage{amsfonts, amsmath, amssymb, amscd, amsthm, bm, cancel,mathrsfs}
\usepackage{url}
\usepackage{graphicx}
\usepackage[
linktocpage=true,colorlinks,citecolor=magenta,linkcolor=blue,urlcolor=magenta]{hyperref}
\usepackage{multicol}
\usepackage{comment}
\usepackage{tikz}
\usepackage[margin=1.1in]{geometry}
\newtheorem{thm}{Theorem}[section]
\newtheorem*{thmA}{Theorem A}
\newtheorem*{thmA'}{Theorem A'}

\newtheorem{lem}[thm]{Lemma}

\newtheorem{prop}[thm]{Proposition}

\theoremstyle{definition}

\newtheorem*{ack}{Acknowledgments}
\theoremstyle{remark}
\newtheorem{rem}[thm]{Remark}

\numberwithin{equation}{section}
\allowdisplaybreaks

\renewcommand{\(}{\left(}
\renewcommand{\)}{\right)}
\renewcommand{\~}{\tilde}
\renewcommand{\-}{\bar}

\renewcommand{\a}{\alpha}
\renewcommand{\b}{\beta}
\newcommand{\g}{\gamma}
\renewcommand{\d}{\delta}
\newcommand{\e}{\varepsilon}
\renewcommand{\k}{\kappa}

\newcommand{\D}{\Delta}

\renewcommand{\t}{\theta}
\newcommand{\s}{\sigma}

\newcommand{\G}{\Gamma}

\newcommand{\ra}{\rightarrow}

\begin{document}
\title[On the mean curvature type flow]{On the mean curvature type flow for convex capillary hypersurfaces in the ball}
\author[Y. Hu]{Yingxiang Hu}
\address{School of Mathematics, Beihang University, Beijing 100191, P.R. China}
\email{\href{mailto:huyingxiang@buaa.edu.cn}{huyingxiang@buaa.edu.cn}}
\author[Y. Wei]{Yong Wei}
\address{School of Mathematical Sciences, University of Science and Technology of China, Hefei 230026, P.R. China}
\email{\href{mailto:yongwei@ustc.edu.cn}{yongwei@ustc.edu.cn}}
\author[B. Yang]{Bo Yang}
\address{Institute of Mathematics, Academy of Mathematics and Systems Sciences, Chinese Academy of Sciences,
Beijing, 100190, P. R. China}
\email{\href{mailto:boyang16@amss.ac.cn}{boyang16@amss.ac.cn}}
\author[T. Zhou]{Tailong Zhou}
\address{School of Mathematical Sciences, University of Science and Technology of China, Hefei 230026, P.R. China}
\email{\href{mailto:ztl20@ustc.edu.cn}{ztl20@ustc.edu.cn}}

\begin{abstract}
In this paper, we study the mean curvature type flow for hypersurfaces in the unit Euclidean ball with capillary boundary, which was introduced by Wang-Xia \cite{Wang-Xia2019} and Wang-Weng \cite{WW2020}. We show that if the initial hypersurface is strictly convex, then the solution of this flow is strictly convex for $t>0$, exists for all positive time and converges smoothly to a spherical cap. As an application, we prove a family of new Alexandrov-Fenchel inequalities for convex hypersurfaces in the unit Euclidean ball with capillary boundary.
\end{abstract}	

\maketitle
\tableofcontents

\section{Introduction}
In this paper, we are interested in the mean curvature type flow for hypersurfaces  in the unit Euclidean ball with capillary boundary, which was introduced recently by Wang-Xia \cite{Wang-Xia2019} and Wang-Weng \cite{WW2020}. We first describe some definitions and notations for such hypersurfaces. The readers are referred to \cite{Weng-Xia2022} for more details.  Let $\Sigma\subset\bar{\mathbb{B}}^{n+1} (n\geq 2)$ be a properly embedded smooth hypersurface in the unit Euclidean ball $\bar{\mathbb{B}}^{n+1}$, given by an embedding $x:\bar{\mathbb{B}}^n\to\bar{\mathbb{B}}^{n+1}$ satisfying
\begin{align*}
\Sigma=x(\mathbb{B}^n)\subset \mathbb{B}^{n+1},\quad  \partial\Sigma=x(\partial\mathbb{B}^n)\subset \mathbb{S}^n.
\end{align*}
Let $\-N$ be the unit outward normal of $\mathbb S^n=\partial \mathbb B^{n+1}$, and $\nu$ be a smooth choice of the unit normal of $\Sigma$. For $\theta\in(0,\pi)$, we call that $\Sigma$ has  $\theta$-capillary boundary $\partial\Sigma$ on $\mathbb{S}^n$ if $\Sigma$ intersects $\mathbb S^n$ at the constant contact angle $\t$, that is,
\begin{equation*}
\langle\bar{N}\circ x,\nu\rangle =-\cos\theta, \quad \text{along $\partial \mathbb B^n$}.
\end{equation*}
In particular, if $\theta=\frac{\pi}{2}$, i.e., $\Sigma$ intersects $\mathbb S^n$ orthogonally, we call that $\Sigma$ has free boundary. Two model examples are the spherical cap of radius $r$ around a constant unit vector $e\in \mathbb{S}^n\subset\mathbb R^{n+1}$ with $\t$-capillary boundary, given by
\begin{align}\label{s1:spherical-cap}
C_{\t,r}(e):=\left\{x\in \-{\mathbb B}^{n+1}:\left| x-\sqrt{r^2+2r\cos\t+1}e\right|=r\right\},
\end{align}
and the {flat ball} around $e$ with $\t$-capillary boundary, given by
\begin{align}\label{s1:flat-disk}
C_{\t,\infty}(e):=\{ x\in \-{\mathbb B}^{n+1}:\langle x,e\rangle=\cos\t\}.
\end{align}

Denote the principal curvatures of $\Sigma$ by $\kappa=(\kappa_1,\cdots,\kappa_n)$.  For $k=1,\cdots,n$, we denote by $H_k$ the normalized $k$th mean curvature of $\Sigma$, which is defined as the normalized $k$th elementary symmetric polynomial of $\kappa$:
\begin{equation*}
H_k=\binom{n}{k}^{-1}\s_k(\kappa), \quad k=1,\cdots,n.
\end{equation*}
We also denote the mean curvature as $H=nH_1$. We say that $\Sigma$ is convex (resp. strictly convex) if its principal curvatures $\kappa_i\geq 0$ (resp. $\kappa_i>0$). In this paper, we always assume that $\Sigma$ is convex.  Denote by $\widehat{\partial\Sigma}$ the convex body in $\mathbb{S}^n$ enclosed by $\partial \Sigma$, and $\widehat{\Sigma}$ the convex domain in $\overline{\mathbb{B}}^{n+1}$ enclosed by $\widehat{\partial\Sigma}$ and $\Sigma$. See Figure \ref{fig1}.  We choose the unit normal $\nu$ of $\Sigma$ as the one pointing outward of $\widehat{\Sigma}$. For each $e\in \mathbb{S}^n\subset\mathbb{R}^{n+1}$, the smooth vector field $X_e$ in $\mathbb{R}^{n+1}$ defined by
\begin{equation}\label{eq-Xe}
X_e=\langle x,e\rangle x-\frac{1}{2}(|x|^2+1)e
\end{equation}
is a conformal Killing vector field in $\mathbb{R}^{n+1}$ and its restriction on $\partial\mathbb{B}^{n+1}$ is a tangential vector field on $\partial\mathbb{B}^{n+1}$ (see \cite[Prop. 3.1]{Wang-Xia2019-2}).  Using $X_e$, the following Minkowski type formula (see \cite[Prop. 2.8]{Weng-Xia2022}) holds:
\begin{equation}\label{s1:Minkowski-identity}
\int_{\Sigma}{H_{k-1} \langle x+\cos\theta\nu,e \rangle}\,dA=\int_{\Sigma}{H_k\langle X_e,\nu\rangle}\, dA,\quad k=1,\cdots, n.
\end{equation}

Based on the formula \eqref{s1:Minkowski-identity}, the following mean curvature type flow with $\theta$-capillary boundary was introduced by Wang and Xia \cite{Wang-Xia2019} for the case $\t=\frac{\pi}{2}$, and later by Wang and Weng \cite{WW2020} for general case $\t\in (0,\pi)$:
\begin{equation}\label{Guan-Li-flow}
\left\{\begin{aligned}
\left(\partial_t x\right)^{\bot}&=\Big(n \langle x+\cos\theta\nu,e\rangle-H\langle X_e,\nu\rangle\Big)\nu &\text{in}\quad \bar{\mathbb{B}}^n \times[0,T),\\
\langle\bar{N}\circ x,\nu\rangle&=-\cos\theta  &\text{on}\quad \partial\bar{\mathbb{B}}^n \times[0,T),\\
x(\cdot,0)&=x_0(\cdot)  &\text{on} \quad \bar{\mathbb{B}}^n,
\end{aligned}\right.
\end{equation}
where $(\cdot)^{\bot}$ denotes the projection to the normal bundle of $\Sigma$. Denote $\Sigma_t=x(\bar{\mathbb{B}}^n,t)$ as the solution of the flow \eqref{Guan-Li-flow}. The flow \eqref{Guan-Li-flow} has a property that the volume $|\widehat{\Sigma_t}|$ is preserved which follows from \eqref{s1:Minkowski-identity} for $k=1$. The convergence results \cite{Wang-Xia2019,WW2020} of the flow \eqref{Guan-Li-flow} for star-shaped initial hypersurface $\Sigma$, i.e., $\langle X_e,\nu\rangle >0$ holds everywhere on $\Sigma$, can be stated as follows.
\begin{thmA}
Let $\Sigma\subset\bar{\mathbb{B}}^{n+1}(n\geq 2)$ be a properly embedded smooth hypersurface with $\theta$-capillary boundary, given by an embedding: $x:\bar{\mathbb{B}}^n\to\Sigma\subset\bar{\mathbb{B}}^{n+1}$. Assume that $\Sigma$ is star-shaped with respect to $e\in \mathbb S^n$, i.e., $\langle X_e,\nu\rangle >0$ holds everywhere in $\Sigma$.
\begin{enumerate}[$(i)$]
\item \cite{Wang-Xia2019} If $\t=\frac{\pi}{2}$, then the solution $\Sigma_t$ of the flow \eqref{Guan-Li-flow} exists for all positive time and converges smoothly to a spherical cap with free boundary as $t\ra \infty$, whose enclosed domain has the same volume as $|\widehat{\Sigma}|$.
\item \cite{WW2020} If $\t$ satisfies $|\cos\t|<\frac{3n+1}{5n-1}$, then the solution $\Sigma_t$ of the flow \eqref{Guan-Li-flow} exists for all positive time and subsequently converges smoothly to a spherical cap with $\t$-capillary boundary as $t\ra \infty$, whose enclosed domain has the same volume as $|\widehat{\Sigma}|$.
\end{enumerate}
\end{thmA}

The first aim of this paper is to prove the preservation of the strict convexity along the flow \eqref{Guan-Li-flow}. The main tool is the tensor maximum principle in Theorem \ref{s1:thm-max principle}, which was developed by the authors in \cite{HWYZ2022}.  With help of the convexity preserving, we establish the convergence result of the flow \eqref{Guan-Li-flow} for convex hypersurfaces with $\theta$-capillary boundary in the unit Euclidean ball.
\begin{thm}\label{thm-con}
Let $\Sigma\subset\bar{\mathbb{B}}^{n+1}(n\geq 2)$ be a properly embedded, strictly convex smooth hypersurface with $\theta$-capillary boundary ($\theta\in (0,\frac{\pi}{2}]$), given by an embedding: $x:\bar{\mathbb{B}}^n\to\Sigma\subset\bar{\mathbb{B}}^{n+1}$. Then there exists a constant unit vector $e\in \mathbb S^n$ such that $\langle X_e,\nu\rangle>0$ holds everywhere on $\Sigma$, and the solution $\Sigma_t$  of the flow \eqref{Guan-Li-flow} starting from $\Sigma$ is strictly convex and exists for all time $t>0$. Moreover, $\Sigma_t$ converges smoothly to a spherical cap $C_{\t,r_{\infty}}(e)$ as $t\ra \infty$, where $r_\infty$ is determined by $|\widehat{C_{\t,r_\infty}}(e)|=|\widehat{\Sigma}|$.
\end{thm}

The mean curvature type flow \eqref{Guan-Li-flow} is a locally constrained curvature flow, which was previously considered by Guan and Li \cite{GL-2015} for closed hypersurfaces in space forms. Other kinds of locally constrained curvature flows have been investigated by many authors, see \cite{BGL,GLW-2019,Lambert-Scheuer2021,Sch21,SX-2019,HLW2020,GL-2021,Wei-X2022a,Wei-X2022b}, etc. A main motivation of studying these locally constrained curvature flows is their powerful applications in proving geometric inequalities including the Alexandrov-Fenchel type inequalities for quermassintegrals.

The quermassintegrals for convex hypersurfaces with $\theta$-capillary boundary in $\-{\mathbb B}^{n+1}$ were first introduced by Scheuer, Wang and Xia \cite{Scheuer-Wang-Xia2018} for $\theta=\frac{\pi}2$, and later by Weng and Xia \cite{Weng-Xia2022} for $\t\in(0,\frac{\pi}{2}]$. Let $\Sigma\subset\bar{\mathbb{B}}^{n+1} (n\geq 2)$ be a properly embedded, convex smooth hypersurface with $\theta$-capillary boundary $\partial\Sigma\subset\mathbb{S}^n$, where $\t\in (0,\frac{\pi}{2}]$.  Then the quermassintegrals $W_{k,\theta}(\widehat{\Sigma})$ for $\widehat{\Sigma}$ are defined by \eqref{s2:quermassintegral-0}--\eqref{s2:quermassintegral-2}. In particular,
\begin{align*}
W_{0,\theta}(\widehat{\Sigma})=&|\widehat{\Sigma}|, \qquad W_{1,\theta}(\widehat{\Sigma})=\frac{1}{n+1}\left(|\Sigma|-\cos\theta|\widehat{\partial\Sigma}|\right).
\end{align*}
For strictly convex hypersurfaces  in $\-{\mathbb B}^{n+1}$ with $\theta$-capillary boundary, the following flow
\begin{equation}\label{inverse-harmonic-mean-curvature-flow}
\left\{\begin{aligned}
\left(\partial_t x\right)^{\bot}&=\(\frac{\langle x,e\rangle +\cos\t \langle \nu,e\rangle}{H_n/H_{n-1}}-\langle X_e,\nu\rangle\)\nu  &\text{in}\quad \bar{\mathbb{B}}^n \times[0,T),\\
\langle\bar{N}\circ x,\nu\rangle&=-\cos\theta &\text{on}\quad \partial\bar{\mathbb{B}}^n \times[0,T),\\
x(\cdot,0)&=x_0(\cdot)  &\text{on} \quad \bar{\mathbb{B}}^n,
\end{aligned}\right.
\end{equation}
was studied in \cite{Scheuer-Wang-Xia2018,Weng-Xia2022}, and as an application they proved that for any $\ell=0,1,\cdots,n-1$,  there holds
\begin{equation}\label{eq-AF-n-k}
W_{n,\theta}(\widehat{\Sigma})\geq (f_n\circ f_\ell^{-1})(W_{\ell,\theta}(\widehat{\Sigma})),
\end{equation}
where $f_\ell(r)$ is the strictly increasing function given by $f_\ell(r)=W_{\ell,\theta}(\widehat{C_{\theta,r}})$. 	
Equality holds in \eqref{eq-AF-n-k} if and only if $\Sigma$ is a spherical cap or a flat ball with $\theta$-capillary boundary.

As an application of our Theorem \ref{thm-con}, we prove new sharp Alexandrov-Fenchel type inequalities relating $W_{k,\theta}(\widehat{\Sigma})$ and  $W_{0,\theta}(\widehat{\Sigma})$ for $k\geq 1$.
\begin{thm}\label{co-af}
Let $\Sigma\subset\bar{\mathbb{B}}^{n+1} (n\geq 2)$ be a properly embedded, convex smooth hypersurface in the unit ball with $\theta$-capillary boundary, where $\t\in (0,\frac{\pi}{2}]$. Then for any $k=1,\cdots, n-1$,  there holds
	\begin{equation}\label{eq-AF}
	W_{k,\theta}(\widehat{\Sigma})\geq (f_k\circ f_0^{-1})(W_{0,\theta}(\widehat{\Sigma})).
	\end{equation}
Equality holds in \eqref{eq-AF} if and only if $\Sigma$ is a spherical cap or a flat ball with $\theta$-capillary boundary.
\end{thm}
\begin{rem}
The inequality \eqref{eq-AF} for $k=n$ is the inequality \eqref{eq-AF-n-k} for $\ell=0$ and has been proved in \cite{Scheuer-Wang-Xia2018,Weng-Xia2022}.
\end{rem}

The paper is organized as follows: In \S \ref{sec:2}, we give some preliminaries for convex capillary hypersurfaces in the unit ball, including some estimates on the geometric quantities on the hypersurfaces. We also recall the definition of quermassintegrals for such hypersurfaces. In \S \ref{sec:3}, the evolution equations along the flow \eqref{Guan-Li-flow} will be deduced. In \S \ref{sec:5}, we apply the tensor maximum principle to prove the preservation of convexity along the flow \eqref{Guan-Li-flow}. In \S \ref{sec:6}, we give the proofs of Theorem \ref{thm-con} and Theorem \ref{co-af}.

\begin{ack}
The authors would like to thank Professor Chao Xia for his helpful suggestions. The research was surpported by National Key Research and Development Program of China 2021YFA1001800 and 2020YFA0713100, National Natural Science Foundation of China NSFC11721101 and NSFC12101027, and Research grant KY0010000052 from University of Science and Technology of China.
\end{ack}

\section{Convex hypersurfaces with capillary boundary}\label{sec:2}

In this section, we collect some preliminaries on smooth hypersurfaces in the unit ball with capillary boundary, the quermassintegrals, and some estimates on convex hypersurfaces with capillary boundary.
\subsection{Hypersurfaces in the ball with capillary boundary}

Let $\t\in (0,\frac{\pi}{2}]$. Suppose that $\Sigma\subset\bar{\mathbb{B}}^{n+1}$ is a smooth, properly embedded, convex hypersurface with $\theta$-capillary boundary, which is given by an embedding $x:\-{\mathbb B}^n\ra \Sigma$ such that
$$
\Sigma=x(\mathbb B^{n}) \subset \mathbb B^{n+1}, \quad \partial \Sigma=x(\partial \mathbb B^{n})\subset \mathbb S^n.
$$
Then $\partial\Sigma\subset\mathbb{S}^{n}$ is a convex hypersurface of the unit sphere and bounds a convex body in $\mathbb{S}^n$, which we denote by $\widehat{\partial\Sigma}$, cf. \cite[Theorem 1.1]{CW1970}. Denote by $\widehat{\Sigma}$ the bounded domain in $\-{\mathbb{B}}^{n+1}$ enclosed by $\Sigma$ and $\widehat{\partial\Sigma}$. We denote by $\nu$ the unit normal field of $\Sigma$ and $\bar{N}$ the position vector of $\mathbb{S}^n$. We identify the outward pointing conormal of $\partial\Sigma\subset\Sigma$ with $\mu$ and the unit normal of $\partial \Sigma$ in $\mathbb{S}^n$ with $\bar{\nu}$ such that $\{\nu,\mu\}$ and $\{\bar{\nu},\bar{N}\circ x\}$ have the same orientation in the normal bundle of $\partial\Sigma\subset\mathbb{S}^n$, see Figure \ref{fig1}.
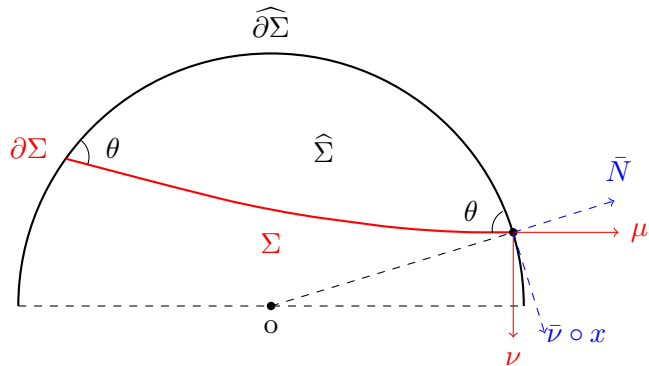
\begin{figure}[htbp]
	\begin{tikzpicture}[scale=1.4,line width=0.4pt]
	\draw[thick] (2.4,0) arc (0:180:2.4);
	\draw[dashed](-2.4,0)--(2.4,0);
	\filldraw[thick](0,0) circle(0.03);
	\node at (0,-0.2) {o};
	\node at (0,0.6) {\textcolor{red}{$\Sigma$}};
	\node at (0,2.7) {$\widehat{\partial\Sigma}$};
	\node at (0.5,1.5) {$\widehat{\Sigma}$};
	
	\draw[red,thick](-1.95,1.4) [rounded corners=10pt] -- (-1,1.15)
	[rounded corners=20pt]--(0,0.9)
	[rounded corners=20pt]--(1.5,0.7)
	[rounded corners=20pt]--(2.3,0.7);
	
	\node at (-2.3,1.5) {\textcolor{red}{$\partial \Sigma$}};
	\filldraw[thick](2.3,0.7) circle(0.03);
	
	\draw (2.2,0.9) arc (120:190:0.2);
	\node at (1.9,0.9) {$\theta$};
	\draw (-1.8,1.57) arc (50:-20:0.2);
	\node at (-1.5,1.5) {$\theta$};
	
	\draw[red,->] (2.3,0.7)--(3.3,0.7);
	\node at (3.5,0.7) {\textcolor{red}{$\mu$}};
	\draw[red,->] (2.3,0.7)--(2.3,-0.3);
	\node at (2.3,-0.5) {\textcolor{red}{$\nu$}};

	\draw[blue,dashed,->] (2.3,0.7)--(3.26,1);
	\draw[dashed](0,0)--(2.3,0.7);
	\node at (3.3,1.3) {\textcolor{blue}{$\bar{N}$}};
	\draw[blue,dashed,->] (2.3,0.7)--(2.6,-0.26);
	\node at (2.9,-0.26) {\textcolor{blue}{$\bar{\nu}\circ x$}};
	
	\end{tikzpicture}
	\caption{A convex $\t$-capillary hypersurface in the unit ball}\label{fig1}
\end{figure}

It follows from $\langle \-N\circ x,\nu\rangle=-\cos \t$ that
\begin{equation}\label{normal transform}
\left\{\begin{aligned}
\bar{N}\circ x=&\sin\theta\mu-\cos\theta\nu,\\
\bar{\nu}=&\cos\theta\mu+\sin\theta\nu.
\end{aligned}\right.
\end{equation}

We denote by $D$ the Levi-Civita connection of $\mathbb R^{n+1}$ with respect to the Euclidean metric $\d$, and $\nabla$ the Levi-Civita connection on $\Sigma$ with respect to the induced metric from the embedding $x:\-{\mathbb B}^n \ra \Sigma \subset \mathbb R^{n+1}$, respectively. The second fundamental of $\Sigma$ in $\mathbb R^{n+1}$ is given by
\begin{equation*}
h(X,Y):=-\langle D_X Y,\nu\rangle,\ \ X,Y\in T(\Sigma).
\end{equation*}
Note that $\partial\Sigma$ can be viewed as a smooth closed hypersurface in $\mathbb{S}^n$ and $\Sigma$, respectively. The second fundamental form of $\partial\Sigma$ in $\mathbb{S}^n$ is given by
\begin{equation*}
\widehat{h}(X,Y):=-\langle\nabla^{\mathbb{S}^n}_X Y,\bar{\nu}\rangle=-\langle D_XY,\bar{\nu}\rangle,\ \ X,Y\in T(\partial\Sigma).
\end{equation*}
The second fundamental form of $\partial\Sigma$ in $\Sigma$ is given by
\begin{equation*}
\tilde{h}(X,Y):=-\langle\nabla_X Y,\mu\rangle=-\langle D_XY,\mu\rangle,\ \ X,Y\in T(\partial\Sigma).
\end{equation*}

\begin{prop}[\cite{Weng-Xia2022}]\label{boundary-h-property}
	Let $\Sigma\subset\bar{\mathbb{B}}^{n+1}$ be a $\theta$-capillary hypersurface. Let $\{e_\alpha\}_{\alpha=1}^{n-1}$ be an orthonormal frame of $\partial\Sigma$. Then the following relations hold on $\partial \Sigma$:
\begin{enumerate}[$(1)$]
  \item $\mu$ is a principal direction of $\Sigma$, i.e., $h(\mu,e_\alpha)=0$.
  \item $h_{\alpha\beta}=\sin\theta\widehat{h}_{\alpha\beta}-\cos\theta\delta_{\alpha\beta}$.
  \item $\tilde{h}_{\alpha\beta}=\cos\theta\widehat{h}_{\alpha\beta}+\sin\theta\delta_{\alpha\beta}=\cot\theta h_{\alpha\beta}+\frac{1}{\sin\theta}\delta_{\alpha\beta}$.
  \item $\nabla_\mu h_{\alpha\beta}=\tilde{h}_{\beta\gamma}\left(h_{\mu\mu}\delta_{\alpha\gamma}-h_{\alpha\gamma}\right)$.
\end{enumerate}
\end{prop}

\subsection{Quermassintegrals}
The quermassintegrals for $\theta$-capillary hypersurfaces in the unit ball were first introduced for $\t=\frac{\pi}{2}$ by Scheuer, Wang and Xia \cite{Scheuer-Wang-Xia2018}, and later generalized by Weng and Xia \cite{Weng-Xia2022} to $\t\in (0,\frac{\pi}{2}]$. Let $\Sigma\subset\bar{\mathbb{B}}^{n+1}$ be a smooth convex, embedded hypersurface with $\theta$-capillary boundary $\partial\Sigma\subset\mathbb{S}^n$. The quermassintegrals for $\widehat{\Sigma}$ are defined by
\begin{align}
W_{0,\theta}(\widehat{\Sigma})=&|\widehat{\Sigma}|, \label{s2:quermassintegral-0}\\
W_{1,\theta}(\widehat{\Sigma})=&\frac{1}{n+1}\left(|\Sigma|-\cos\theta W_0^{\mathbb{S}^n}(\widehat{\partial\Sigma})\right),
\end{align}
and for $1\leq k\leq n-1$,
\begin{align}\label{s2:quermassintegral-2}
\begin{aligned}
&W_{k+1,\theta}(\widehat{\Sigma})=\frac{1}{n+1}\left\{\int_{\Sigma}H_{k}dA-\cos\theta\sin^k\theta W_k^{\mathbb{S}^n}(\widehat{\partial\Sigma})\right.  \\
&\quad \quad \quad \left.-\cos^{k-1}\theta\sum_{\ell=0}^{k-1}\frac{(-1)^{k+\ell}}{n-\ell}\binom{k}{\ell}\left[(n-k)\cos^2\theta+k-\ell\right]\tan^\ell\theta W_\ell^{\mathbb{S}^n}(\widehat{\partial\Sigma})\right\}.
\end{aligned}
\end{align}
For the free boundary case $\theta=\pi/2$, the definition \eqref{s2:quermassintegral-2} simplifies as
\begin{align}\label{s2:quermassintegral-3}
&W_{k+1,\frac{\pi}{2}}(\widehat{\Sigma})=\frac{1}{n+1}\left\{\int_{\Sigma}H_{k}dA-\frac{k}{n-k+1} W_{k-1}^{\mathbb{S}^n}(\widehat{\partial\Sigma})\right).
\end{align}
Here $dA$ is the area element and $H_k$ is the $k$th normalized mean curvature of the hypersurface $\Sigma$ in $\mathbb R^{n+1}$ respectively, and for a $k$-dimensional submanifold $M\subset \mathbb R^{n+1}$ (with or without boundary), $|M|$ denotes the $k$-dimensional Hausdorff measure of $M$. The quermassintegrals $W_{k}^{\mathbb{S}^n}(\widehat{\partial \Sigma})$ in $\mathbb{S}^n$ are defined by (see \cite{Sol06,ChenGLS22})
\begin{align*}
W_0^{\mathbb S^n}(\widehat{\partial \Sigma})=&|\widehat{\partial \Sigma}|, \\
W_1^{\mathbb S^n}(\widehat{\partial \Sigma})=&\frac{1}{n}|\partial \Sigma|,\\
W_{k+1}^{\mathbb S^n}(\widehat{\partial \Sigma})=&\frac{1}{n}\int_{\partial \Sigma} H_{k}^{\mathbb S^n} ds+\frac{k}{n-k+1}W^{\mathbb S^n}_{k-1}(\widehat{\partial \Sigma}), \quad 1\leq k\leq n-1,
\end{align*}
where $ds$ is the area element and $H_{k}^{\mathbb S^n}$ is the $k$th normalized mean curvature of the hypersurface $\partial \Sigma$ in $\mathbb S^n$ respectively.

The quermassintegrals $W_{k,\theta}(\widehat{\Sigma})$ defined in \eqref{s2:quermassintegral-0} - \eqref{s2:quermassintegral-2} are viewed as a natural counterparts to the quermassintegrals for smooth closed hypersurfaces in $\mathbb{R}^{n+1}$, as evidenced by the following nice variational formula:
\begin{prop}[\cite{Weng-Xia2022}]
	Let $\Sigma_t\subset\bar{\mathbb{B}}^{n+1}$ be a family of smooth, properly embedded hypersurfaces with $\theta$-capillary boundary on $\mathbb{S}^n$, given by $x(\cdot,t):\bar{\mathbb{B}}^n\rightarrow \Sigma_t\subset\bar{\mathbb{B}}^{n+1}$, such that
	\begin{equation*}
	\left(\partial_tx\right)^\perp=f\nu
	\end{equation*}
	for some normal speed $f$. Then there holds
	\begin{equation}\label{s2.dW}
	\frac{d}{dt}W_{k,\theta}(\widehat{\Sigma}_t)=\frac{n+1-k}{n+1}\int_{\Sigma_t}H_kfdA, \quad 0\leq k\leq n.
	\end{equation}	
\end{prop}

Using the Minkowski formula \eqref{s1:Minkowski-identity} and the variational formula \eqref{s2.dW}, we have the following monotonicity of quermassintegrals along the flow \eqref{Guan-Li-flow}.
\begin{lem}\label{prop-mono}
Let $\Sigma_t$ be a smooth strictly convex solution of the flow \eqref{Guan-Li-flow}. Then along the flow, $W_{0,\theta}(\widehat{\Sigma_t})$ is preserved and $W_{k,\theta}(\widehat{\Sigma_t})$ is non-increasing in time for all $1\leq k\leq n-1$. Moreover,  $\frac{d}{dt}W_{k,\theta}(\widehat{\Sigma_t})=0$ if and only if $\Sigma_t$ is a spherical cap.
\end{lem}
\begin{proof}
Firstly, by the variational formula \eqref{s2.dW} and the Minkowski formula \eqref{s1:Minkowski-identity} for $k=1$, we have
	\begin{align*}
	\frac{d}{dt}W_{0,\theta}(\widehat{\Sigma_t})=\frac{n}{n+1}\int_{\Sigma_t}\Big(\langle x+\cos\theta\nu,e\rangle-H_1\langle X_e,\nu\rangle \Big)dA=0.
	\end{align*}
As the solution $\Sigma_t$ is strictly convex for $t>0$, using the Newton's inequality $H_{k+1}\leq H_kH_1$ and the Minkowski formula \eqref{s1:Minkowski-identity} again, we obtain
	\begin{align}\label{s5:monotonicity}
	\frac{d}{dt}W_{k,\theta}(\widehat{\Sigma_t})=&\frac{n(n+1-k)}{n+1}\int_{\Sigma_t}\Big(H_k\langle x+\cos\theta\nu,e\rangle-\langle X_e,\nu\rangle H_1H_k\Big)dA \nonumber\\
	=&\frac{n(n+1-k)}{n+1}\int_{\Sigma_t}\left(H_{k+1}-H_1H_k\right)\langle X_e,\nu\rangle dA \nonumber\\
	\leq&0,
	\end{align}
for $k=1,\cdots,n-1$. If equality holds in \eqref{s5:monotonicity}, then $\Sigma_t$ are umbilic and hence it is a spherical cap with $\t$-capillary boundary.
\end{proof}

\subsection{Estimates on convex capillary hypersurfaces}
In this subsection, we collect some estimates on strictly convex hypersurfaces with $\t$-capillary boundary in the unit ball, which were obtained by Weng and Xia \cite{Weng-Xia2022}.
\begin{prop}\cite[Prop. 2.15]{Weng-Xia2022}
Let $\Sigma\subset\bar{\mathbb{B}}^{n+1}$ be a strictly convex hypersurface with $\theta$-capillary boundary for $\theta\in (0,\frac{\pi}{2}]$. Then for any $p\in \Sigma$, $\Sigma$ lies on one side of $T_p\Sigma$.
\end{prop}

Let $\Sigma\subset\bar{\mathbb{B}}^{n+1}$ be a strictly convex hypersurface with $\theta$-capillary boundary.  Then it is proved in \cite[Prop. 2.16]{Weng-Xia2022} that there exists a point $e\in \operatorname{int}(\widehat{\partial\Sigma})$ and a constant $0<\d_0<1-\cos\t$ depending only on $\Sigma$ such that
		\begin{align}
		& \langle x,e\rangle \geq \cos \t+\d_0, \label{s2.est1}\\
		& \langle x,\nu\rangle \leq -\cos\t.\label{s2.est2}
		\end{align}
Recall that the foliations $C_{\t,s}(e)$, $s\geq 0$ are spherical caps of radius $s$ around $e$ with $\t$-capillary boundary such that $C_{\t,\infty}(e)$ is the flat ball around $e$ with $\t$-capillary boundary. Then there exists some $0<R_1<R_2<\infty$ such that
\begin{align*}
\Sigma \subset \widehat{C_{\t,R_2}(e)} \backslash \widehat{C_{\t,R_1}(e)}.
\end{align*}

\begin{prop}[cf. Prop.2.16 of \cite{Weng-Xia2022}]\label{s2:prop-estimate-convex}
	Let $\Sigma\subset\bar{\mathbb{B}}^{n+1}$ be a strictly convex hypersurface with $\theta$-capillary boundary such that $\Sigma \subset \widehat{C_{\t,R_2}(e)} \backslash \widehat{C_{\t,R_1}(e)}$. Then the following estimates hold:
\begin{enumerate}[$(1)$]
  \item $\langle x-e,\nu\rangle \geq \d_1>0$;
  \item $\langle X_e,\nu\rangle \geq \d_2>0$;
  \item $\cos\t+\d_0\leq \langle x,e \rangle \leq 1-\d_3$;
  \item $\langle e,\nu\rangle\leq -\d_4$.
\end{enumerate}
	Here $\d_i$, $i=1,2,3,4$ are positive constants depending only on $R_1$, $\t$, and $\d_0$ is a positive constant depending only on $R_2$, $\t$ such that $\cos\t+\d_0\leq 1-\d_3$.
\end{prop}
\begin{proof}
The estimates are proved in Prop.2.16 of \cite{Weng-Xia2022}. Here we provide a proof by clarifying that   the constants  $\delta_i,i=0,1,\cdots,4$ can be chosen explicitly depending only on $\theta, R_1$ and $R_2$. These estimates will be used in the proof of preserving of the strict convexity of $\Sigma_t$ in Theorem \ref{s4:thm-strict-convex}.
	Let $S$ be a flat ball in $\-{\mathbb B}^{n+1}$ such that $\partial S=\partial C_{\t,R_1}(e)$. Since $S\subset \widehat{C_{\t,R_1}(e)}\subset\widehat{\Sigma}$, we know that $S$ lies in the upper half-space determined by $T_x\Sigma$ in the direction of $-\nu$, see Figure \ref{fig3}.
	\begin{figure}[htbp]
		\begin{tikzpicture}[scale=1.4]
		\draw[thick] (2.4,0) arc (0:180:2.4);
		\draw[thick] (-2.25,0.8) arc (-115:-65:5.35);
		\filldraw[thick](0,2.4) circle(0.03);
		\node at (0.2,2.6) {$e$};
		\draw[thick] (-0.923,2.215) arc (199:341:1);	
		
		\draw[red,thick](-0.923,2.2)--(0.923,2.2);
		\filldraw[thick](-0.923,2.2) circle(0.03);
		\node at (0.3,2.0) {$S$};
		\node at (1.2,1.5) {$C_{\theta,R_1}(e)$};
		\node at (1.6,0.2) {$C_{\theta,R_2}(e)$};
		\draw[thick](-3.0,0)--(3.0,0);
		\filldraw[thick](0,0) circle(0.03);
		\node at (0.2,-0.2) {$O$};
		\draw[red,thick](-1.7,1.7) [rounded corners=10pt] -- (-1.3,1.3)
		[rounded corners=20pt]--(0,0.7)
		[rounded corners=20pt]--(1.5,0.6)
		[rounded corners=20pt]--(2.2,1);
		
		\draw[blue](-2.3,1.9)--(0.3,0.38);
		\draw[blue](-2,2.8)--(0.8,1.2);
		\node at (-2.3,2.1) {$T_x\Sigma$};
		\filldraw[thick](-0.32,1.85) circle(0.03);
		\draw[dashed](0,2.4)--(-0.32,1.85);
		\draw[dashed](-0.32,1.85)--(-0.8,1);
		\filldraw[thick](-0.8,1) circle(0.03);
		\filldraw[thick](-1.05,1.18) circle(0.03);
		
		\node at (-0.85,1.38) {$x$};
		
		
		\draw[->][thick](-1.05,1.18)--(-1.3,0.7);
		\node at (-1.1,0.7) {$\nu$};
		
		\node at (1.5,0.9) {$\Sigma$};
		\end{tikzpicture}
		\caption{\ Capillary hypersurface between two spherical caps}\label{fig3}
	\end{figure}
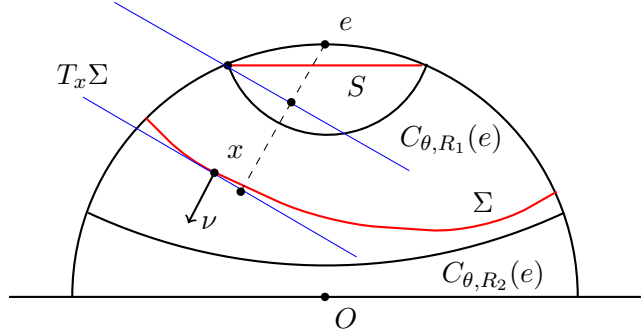
	Let $d^{\mathbb R}$ be the distance in $\mathbb R^{n+1}$. Since $\langle x-e,\nu\rangle$ measures the distance from $e$ to $T_x \Sigma$, which can be bounded from below by the distance from $e$ to $S$, i.e., $\langle x-e,\nu\rangle \geq d^{\mathbb R}(e,S)$. We calculate $d^{\mathbb R}(e,S)$ explicitly as follows:
	\begin{align*}
	1-(1-d^{\mathbb R}(e,S))^2=R_1^2-(\sqrt{R_1^2+2R_1\cos\t+1}-1+d^{\mathbb R}(e,S))^2,
	\end{align*}
	which gives
	$$
	d^{\mathbb R}(e,S)=1-\frac{1+R_1\cos\t}{\sqrt{R_1^2+2R_1\cos\t+1}}=: \d_1(R_1,\t)>0.
	$$
	Then we get $\langle x-e,\nu \rangle \geq \d_1$. It follows that
	\begin{align*}
	\langle X_e,\nu \rangle=\frac{1+|x|^2}{2}\langle x-e,\nu \rangle-\langle x,\nu\rangle\frac{|e-x|^2}{2}
	\geq \frac{1}{2}\d_1=:\d_2(R_1,\t),
	\end{align*}	
	where we used $\langle x,\nu\rangle \leq -\cos\t\leq 0$ by \eqref{s2.est2}.
	
	Note that $\Sigma \subset \widehat{C_{\t,R_2}(e)}$, i.e., $\Sigma$ lies above  $C_{\t,R_2}(e)$ in the direction of $e$. Then we get
	\begin{align*}
	\min_{\Sigma}\langle x,e\rangle \geq &\min_{C_{\t,R_2}(e)}\langle x,e\rangle=\sqrt{R_2^2+2R_2\cos\t+1}-R_2\\
	=&\cos\t+\frac{\sin^2\t}{\sqrt{R_2^2+2R_2\cos\t+1}+R_2+\cos\t}\\
	=:&\cos\t+\d_0(R_2,\t).
	\end{align*}
Similarly,
\begin{align*}
  \max_\Sigma \langle x,e\rangle \leq  & \max_{C_{\theta,R_1}(e)}\langle x,e\rangle =\frac{1+R_1\cos\theta}{\sqrt{R_1^2+2R_1\cos\theta+1}} \\
   =:& 1-\delta_3(R_1,\theta).
\end{align*}

	Since $\partial C_{\t,R_1}(e)$ is enclosed by $\partial \Sigma$ in $\mathbb S^n$, we know that $\nu|_{\partial\Sigma}$ is strictly contained in $B_{\frac{\pi}{2}-\b}(-e)$ in $\mathbb S^n$, where $\b$ is given by
	\begin{equation*}
	\cos\b=\frac{1+R_1\cos\t}{\sqrt{R_1^2+2R_1\cos\t+1}}.
	\end{equation*}
	The convexity of $\Sigma$ implies that $\nu|_{\Sigma}$ also lies in $B_{\frac{\pi}{2}-\b}(-e)$ and hence
	$$
	\langle \nu,e\rangle \leq -\sin\b =-\frac{R_1\sin\t}{\sqrt{R_1^2+2R_1\cos\t+1}}=-
	\d_4(R_1,\t).
	$$
\end{proof}

\section{Evolution equations}\label{sec:3}
In this section, we derive the evolution equations along the flow \eqref{Guan-Li-flow}. Denote the normal speed of the flow by
\begin{equation}\label{s3.f}
  f=n \langle x+\cos\theta\nu,e\rangle-H\langle X_e,\nu\rangle,
\end{equation}
 where we choose $e\in\text{int}(\widehat{\partial\Sigma_0})$ as in Lemma \ref{s2:prop-weak-convex-rigidity}.  As in \cite[\S 2.4]{Weng-Xia2022}, the $\theta$-capillary boundary condition along the flow requires the tangential component $V=(\partial_tx)^\top\in T\Sigma_t$ must satisfy $V\big|_{\partial\Sigma_t}=f\cot\theta\mu+\tilde{V}$,  where $\tilde{V}\in T\left(\partial\Sigma_t\right)$. Hence up to a time-dependent diffeomorphism of $\partial\mathbb{B}^n$, we can assume that $\~V=0$ and hence
\begin{equation}\label{s3.T}
V\big|_{\partial\Sigma_t}=f\cot\theta\mu.
\end{equation}
The $\theta$-capillary boundary condition also ensures that (see \cite[Prop. 2.12]{Weng-Xia2022})
\begin{equation}\label{s2:bdry-condition-general}
\nabla_\mu f=\left(\frac{1}{\sin\theta}+\cot\theta h_{\mu\mu}\right)f\ \ \text{along}\ \partial\Sigma_t.
\end{equation}

So we rewrite the flow \eqref{Guan-Li-flow} equivalently as the following form:
\begin{equation}\label{s3:GL-flow}
\left\{\begin{aligned}
\partial_t x&=f\nu+V \qquad &\text{in $\bar{\mathbb{B}}^n \times[0,T)$},\\
\langle\bar{N}\circ x,\nu\rangle&=-\cos\theta &\text{on  $\partial\bar{\mathbb{B}}^n \times[0,T)$},\\
x(\cdot,0)&=x_0(\cdot) &\text{on $\bar{\mathbb{B}}^n$},
\end{aligned}\right.
\end{equation}
where $f$ is given by \eqref{s3.f} and the tangential component $V=(\partial_tx)^\top\in T\Sigma_t$ satisfies \eqref{s3.T} on the boundary $\partial\Sigma_t$. By Proposition \ref{s2:prop-estimate-convex}, we have $\langle X_e,\nu\rangle\geq \d_2>0$ on $\Sigma_0$. Then the  short time existence of the flow \eqref{s3:GL-flow} follows from rewriting it as a strictly parabolic scalar equation with oblique boundary condition, using M\"obius coordinate transformation (see \cite[\S 3]{WW2020} for details).

It is easy to check that the spherical cap $C_{\t,r}(e)$ of radius $r$ satisfies
\begin{align*}
\langle x,e\rangle+\cos\t\langle \nu,e\rangle=\frac{1}{r}\langle X_e,\nu \rangle,
\end{align*}
and hence it is stationary along the flow \eqref{s3:GL-flow}. Using the spherical caps as barriers gives the following $C^0$ estimate.
\begin{lem}\label{s3:lem-barrier}
	Assume that the initial hypersurface $\Sigma$ satisfies
	\begin{equation*}
	\Sigma\subset\widehat{C_{\t,R_2}(e)}\backslash\widehat{C_{\t,R_1}(e)},
	\end{equation*}
	for some $0<R_1<R_2<\infty$. Then this property also holds for the solution $\Sigma_t$ along the flow \eqref{s3:GL-flow} for $t>0$.
\end{lem}

Let $g_{ij}$, $h_{ij}$ and $H=g^{ij}h_{ij}$ be the induced metric, the second fundamental form and the mean curvature of the flow hypersurface $\Sigma_t$. We recall the following general evolution equations:
\begin{prop}\label{s3:lem-evol-general-speed}\cite[Proposition 2.11]{Weng-Xia2022}
Along a general flow satisfying
\begin{equation*}
  \partial_tx=f\nu+V,
\end{equation*}
where $V\in T\Sigma_t$, we have
 \begin{align}
 \partial_t g_{ij}=&2fh_{ij}+\nabla_i V_j+\nabla_jV_i,\label{s3.evl-g}\\
 \partial_t h_{ij}=&-\nabla_i\nabla_jf+f\left(h^2\right)_{ij}+\nabla_Vh_{ij}+h_j^k\nabla_iV_k+h_i^k\nabla_jV_k,\label{s3.evl-h}\\
 \partial_t H=&-\Delta f-|A|^2f+\nabla_VH,\label{s3.evl-H}
 \end{align}
 where $(h^2)_{ij}=h_i^kh_{kj}$ and $|A|^2=g^{ij}h_i^kh_{kj}$.
\end{prop}

We also need the following lemma.
\begin{lem}[\cite{Scheuer-Wang-Xia2018}]\label{s2:lem-<X,v>}
	There hold
	\begin{equation}\label{eq-gradient <X,v>}
	\nabla_i\langle X_e,\nu\rangle=\langle e_i,e\rangle\langle x,\nu\rangle-\langle e_i,x\rangle\langle e,\nu\rangle+h_i^k\langle X_e, e_k\rangle
	\end{equation}
	and
	\begin{align}
	\nabla_i\nabla_j\langle X_e,\nu\rangle=&\langle X_e,\nabla h_{ij}\rangle+\langle x,e\rangle h_{ij}-\left(h\right)^2_{ij}\langle X_e,\nu\rangle-g_{ij}\langle e,\nu\rangle\nonumber\\
	&+h_j^k\left(\langle  e_i,e\rangle\langle x,e_k\rangle-\langle e_i,x\rangle\langle e,e_k\rangle\right)\nonumber\\
&+h_i^k\left(\langle e_j,e\rangle\langle x,e_k\rangle-\langle e_j,x\rangle\langle e,e_k\rangle\right).\label{eq-Hess <X,v>}
	\end{align}
\end{lem}

Now, we calculate the evolution equations along the flow \eqref{s3:GL-flow}.
\begin{lem}
	Along the flow \eqref{s3:GL-flow} with $f$ given by \eqref{s3.f}, we have
	\begin{enumerate}[(i)]
    \item The induced metric $g_{ij}$ evolves by
    \begin{align}\label{metric-upper}
    \frac{\partial}{\partial t} g_{ij}=2\left(n\langle x+\cos\theta\nu,e\rangle-H\langle X_e,\nu\rangle\right)h_{ij}+\nabla_i V_j+\nabla_jV_i.
    \end{align}
	\item The mean curvature $H$ evolves by
	\begin{align}\label{H-upper}
	 \frac{\partial}{\partial t} H=&\langle X_e,\nu\rangle \D H+\Big\langle H X_e-n\cos\t e+V,\nabla H \Big\rangle\nonumber\\
&+2\langle \nabla \langle X_e,\nu\rangle,\nabla H\rangle+(H^2-n|A|^2)\langle x,e\rangle.
	\end{align}
	and
	\begin{align}\label{s3:bdry-condition}
	\nabla_\mu H=0, \quad \text{on}~ \partial \Sigma_t.
	\end{align}
	\item The second fundamental form $h=(h_{ij})$ evolves by
	\begin{align}\label{s3:evol-second-fundamental-form}
	\frac{\partial}{\partial t}h_{ij}=& \langle X_e,\nu\rangle \D h_{ij}+\Big\langle HX_e-n\cos \t e+V,\nabla h_{ij}\Big\rangle\nonumber\\
	&+\nabla_i H \(\langle e_j,e\rangle \langle x,\nu \rangle-\langle e_j,x\rangle \langle e,\nu\rangle+h_j^k\langle X_e,e_k\rangle\)\nonumber\\
	&+\nabla_j H  \(\langle e_i,e\rangle \langle x,\nu \rangle-\langle e_i,x\rangle \langle e,\nu\rangle+h_i^k\langle X_e,e_k\rangle\)\nonumber\\
	&+(h^2)_{ij}\Big( n\langle x,e\rangle+2n\cos\t\langle e,\nu\rangle-3\langle X_e,\nu\rangle H\Big)\nonumber\\
	&+h_{ij}\Big( n\langle e,\nu\rangle+\langle X_e,\nu\rangle |A|^2 +H\langle x,e\rangle\Big)\nonumber\\
&+h_j^k \nabla_i V_k+h_i^k\nabla_j V_k -g_{ij}H\langle e,\nu\rangle.
	\end{align}
\end{enumerate}
\end{lem}
\begin{proof}
(1) Equation \eqref{metric-upper} follows from \eqref{s3.f} and \eqref{s3.evl-g}.

(2) Substituting \eqref{s3.f} into \eqref{s3.evl-H}, we get
\begin{align}\label{s3.H1}
\partial_t H=&-n\Delta \langle x,e\rangle -n\cos\theta \Delta \langle \nu,e\rangle \nonumber\\
&+\Delta H\langle X_e,\nu\rangle +2\langle \nabla H,\nabla\langle X_e,\nu\rangle \rangle+H \D \langle X_e,\nu\rangle\nonumber\\
& -|A|^2(n\langle x+\cos\theta\nu,e\rangle-H\langle X_e,\nu\rangle)+\nabla_VH.
\end{align}
Recall that
 \begin{align}
  - \Delta \langle x,e\rangle = & H\langle \nu,e\rangle, \label{s3.H2} \\
   \Delta \langle \nu,e\rangle = & \langle \nabla^i(h_i^ke_k),e\rangle =\langle \nabla H,e\rangle -|A|^2\langle \nu,e\rangle. \label{s3.H3}
 \end{align}
Substituting \eqref{s3.H2}, \eqref{s3.H3} and  \eqref{eq-Hess <X,v>} into \eqref{s3.H1}, we obtain the equation \eqref{H-upper}.

For the boundary condition \eqref{s3:bdry-condition}, by \cite[Proposition 3.3]{Wang-Xia2019-2}, along $\partial \Sigma_t$ we have
\begin{align*}
\nabla_\mu \langle X_e,\nu\rangle=\(\frac{1}{\sin\t}+\cot\t h_{\mu\mu}\)\langle X_e,\nu\rangle,
\end{align*}
and
\begin{align*}
\nabla_\mu \langle x+\cos\t \nu,e\rangle=\(\frac{1}{\sin\t}+\cot\t h_{\mu\mu}\)\langle x+\cos\t \nu,e\rangle.
\end{align*}
Substituting the above two equalities into \eqref{s2:bdry-condition-general}, we obtain $\nabla_\mu H=0$.

(3) Substituting \eqref{s3.f} into \eqref{s3.evl-h}, we have
\begin{align*}
\frac{\partial}{\partial t}h_{ij}=&-n\nabla_i\nabla_j\langle x,e\rangle -n\cos\t\nabla_i\nabla_j\langle \nu,e\rangle \nonumber\\
&+\nabla_i\nabla_jH \langle X_e,\nu\rangle+H\nabla_i\nabla_j\langle X_e,\nu\rangle+\nabla_i H \nabla_j\langle X_e,\nu \rangle+\nabla_j H \nabla_i
\langle X_e,\nu \rangle\\
&+\Big(n\langle x,e\rangle+n\cos\t\langle \nu,e\rangle-H\langle X_e,\nu\rangle\Big)(h^2)_{ij}\\
&+\nabla_V h_{ij}+h_j^k \nabla_i V_k+h_i^k \nabla_j V_k.
\end{align*}
For the first two terms, we have
\begin{align*}
  -\nabla_i\nabla_j\langle x,e\rangle= & h_{ij}\langle \nu,e\rangle,  \\
  \nabla_i\nabla_j\langle \nu,e\rangle= & \nabla_i\langle h_j^ke_k,e\rangle =\langle \nabla h_{ij},e\rangle -(h^2)_{ij}\langle \nu,e\rangle.
\end{align*}
By the Simons' identity
\begin{align*}
\nabla_i\nabla_j H=\D h_{ij}+|A|^2 h_{ij}-H(h^2)_{ij},
\end{align*}
and the equations \eqref{eq-gradient <X,v>} - \eqref{eq-Hess <X,v>}, we get
\begin{align*}
\frac{\partial}{\partial t}h_{ij}=&n h_{ij}\langle e,\nu \rangle+ n \cos\t (h^2)_{ij}\langle e,\nu\rangle-n\cos\t\langle \nabla h_{ij},e\rangle\\
&+\nabla_i H (\langle e_j,e\rangle \langle x,\nu \rangle-\langle e_j,x\rangle \langle e,\nu\rangle+h_j^k\langle X_e,e_k\rangle)\\
&+\nabla_j H  (\langle e_i,e\rangle \langle x,\nu \rangle-\langle e_i,x\rangle \langle e,\nu\rangle+h_i^k\langle X_e,e_k\rangle)\\
&+\langle X_e,\nu\rangle(\D h_{ij}+|A|^2 h_{ij}-H(h^2)_{ij})\\
&+H \(\langle X_e,\nabla h_{ij}\rangle+h_{ij}\langle x,e\rangle-(h^2)_{ij}\langle X_e,\nu\rangle-g_{ij}\langle e,\nu\rangle\)\\
&+Hh_j^k \left(\langle e_i,e\rangle\langle x,e_k\rangle-\langle e_i,x\rangle\langle e,e_k\rangle\right)\\
&+Hh_i^k\left(\langle e_j,e\rangle\langle x,e_k\rangle-\langle e_j,x\rangle \langle e,e_k\rangle\right)\\
&+(n\langle x,e\rangle+n\cos\t\langle \nu,e\rangle-H\langle X_e,\nu\rangle)(h^2)_{ij}\\
&+\nabla_V h_{ij}+h_j^k \nabla_i V_k+h_i^k \nabla_j V_k.
\end{align*}
Then \eqref{s3:evol-second-fundamental-form} follows by rearranging the terms.
\end{proof}

\section{Preserving of convexity}\label{sec:5}
We first recall the following tensor maximum principle which was recently developed by the authors \cite{HWYZ2022} and can be viewed as a generalization of the tensor maximum principles of Hamilton \cite{Hamilton1982}, Stahl \cite{Stahl1996-2} and Andrews \cite{And2007}.
\begin{thm}[\cite{HWYZ2022}]\label{s1:thm-max principle}
	Let $\Sigma$ be a smooth compact manifold with boundary $\partial \Sigma$ and $\mu$ be the outward pointing unit normal vector field of $\partial\Sigma$ in $\Sigma$. Assume that $S_{ij}$ is a smooth time-varying symmetric tensor field on $\Sigma$ satisfying
	\begin{equation*}
	\frac{\partial}{\partial t}S_{ij}=a^{k\ell}\nabla_k\nabla_\ell S_{ij}+b^k\nabla_kS_{ij}+N_{ij}
	\end{equation*}
	on $\Sigma\times [0,T]$, where the coefficients $a^{k\ell}$ and $b^k$ are smooth, $\nabla$ is a (possibly time-dependent) smooth symmetric connection, and $(a^{k\ell})$ is positively definite everywhere. Suppose that
	\begin{equation}\label{conditon-MP}
	N_{ij}\xi^i\xi^j+\sup_{\Gamma}2a^{k\ell}\left(2\Gamma^p_k\nabla_\ell S_{ip}\xi^i-\Gamma_k^p\Gamma_\ell^qS_{pq}\right)\geq 0\quad \mathrm{on}~\Sigma\times (0,T],
	\end{equation}
	\begin{equation}\label{eq-bou}
	\left(\nabla_\mu S_{ij}\right)\xi^i\xi^j\geq 0\qquad \mathrm{on}~ \partial\Sigma\times (0,T]
	\end{equation}
	whenever $S_{ij}\geq 0$ and $S_{ij}\xi^i=0$. If $S_{ij}\geq 0$ everywhere on $\Sigma\times\{0\}$, then it remains so on $\Sigma\times[0,T]$.
\end{thm}
We also need the following useful lemma.
\begin{lem}[\cite{HWYZ2022}]\label{s3:tech-lemma}
Let $\Sigma$ be a smooth compact manifold with boundary $\partial \Sigma$ and $\mu$ be the outward pointing unit normal vector field of $\partial\Sigma$ in $\Sigma$. Suppose that $S_{ij}$ is a smooth symmetric tensor satisfying $S_{ij}\geq 0$ on $\Sigma$ and
\begin{equation}\label{s4.lem1}
  (\nabla_\mu S_{ij})\xi^i\xi^j\geq 0,\qquad\mathrm{ on}~ \partial\Sigma
\end{equation}
whenever $S_{ij}\xi^j=0$ for some tangent vector $\xi$. Define a function
\begin{equation*}
  Z(p,\xi)=S(p)(\xi,\xi)
\end{equation*}
where $p\in \Sigma$ and $\xi\in T_p\Sigma$. If $Z$ attains its minimum at some $(p,\xi)\in T\Sigma$, then $D Z(p,\xi)=0$ and $D^2Z(p,\xi)$ is non-negative definite.
\end{lem}

Theorem \ref{s1:thm-max principle} has been used  in \cite{HWYZ2022} to prove the preservation of convexity along the locally constrained inverse curvature flows for capillary hypersurfaces in the half-space. In the following ,we shall apply Theorem \ref{s1:thm-max principle} to show that the convexity is preserved along the flow \eqref{s3:GL-flow}.

\begin{thm}\label{s4:thm-strict-convex}
   Assume that the initial hypersurface $\Sigma\subset \bar{\mathbb{B}}^{n+1}$ is strictly convex with $\t$-capillary boundary $\partial\Sigma\subset \mathbb{S}^n$ such that $\Sigma \subset \widehat{C_{\t,R_2}(e)} \backslash \widehat{C_{\t,R_1}(e)}$. Let $\Sigma_t, t\in [0,T),$ be the smooth solution of the flow \eqref{s3:GL-flow} starting from $\Sigma$. Then $\Sigma_t$ is strictly convex and there exists a constant $\e>0$ depending only on $\Sigma$ such that the principal curvatures $\kappa_i$ of $\Sigma_t$ satisfy
   \begin{equation*}
     \kappa_i\geq \e,\qquad i=1,\cdots,n
   \end{equation*}
   for all $t\in [0,T)$.
\end{thm}
\begin{proof}
We consider the tensor
	\begin{equation*}
	S_{ij}=h_{ij}-\varepsilon g_{ij},
	\end{equation*}
where $\varepsilon>0$ is a small constant to be determined later. Since the initial hypersurface $\Sigma$ is strictly convex, we can first choose $\varepsilon>0$ small enough such that $S_{ij}\geq 0$ on $\Sigma$. Combining the equations \eqref{metric-upper},  \eqref{s3:evol-second-fundamental-form}, we deduce that
	\begin{align}\label{s5:evol-S_ij-strict-convex}
	\frac{\partial}{\partial t}S_{ij}=&\frac{\partial}{\partial  t}h_{ij}-\e\frac{\partial}{\partial t}g_{ij}   \nonumber\\
	=&\langle X_e,\nu\rangle\Delta S_{ij}+\langle HX_e+n\cos\theta e+V,\nabla S_{ij}\rangle+S_j^k\nabla_iV_k+S_i^k\nabla_jV_k\nonumber\\
	&+\nabla_iH\left(\langle e_j,e\rangle\langle x,\nu\rangle-\langle e_j,x\rangle\langle e,\nu\rangle+h_j^k\langle X_e,e_k\rangle\right)\nonumber\\
	&+\nabla_jH\left(\langle e_i,e\rangle\langle x,\nu\rangle-\langle e_i,x\rangle\langle e,\nu\rangle+h_i^k\langle X_e,e_k\rangle\right)\nonumber\\
	&+((S^2)_{ij}+2\e S_{ij}+\e^2 g_{ij})\left(n\langle x,e\rangle+2n\cos\theta\langle\nu,e\rangle-3\langle X_e,\nu\rangle H\right)\nonumber\\
	&+(S_{ij}+\e g_{ij})(n\langle\nu,e\rangle+\langle X_e,\nu\rangle|A|^2+H\langle x,e\rangle)-g_{ij}H\langle\nu,e\rangle\nonumber\\
	&-2\e(S_{ij}+\e g_{ij})(\langle x+\cos\t \nu,e\rangle-\langle X_e,\nu\rangle H).
	\end{align}

We first check that the boundary condition \eqref{eq-bou} is satisfied whenever $S_{ij}\geq 0$ and $S_{ij}\xi^j=0$ at some point $p_0\in \partial\Sigma_{t_0}$ for some vector $\xi\in T_{p_0}\Sigma_{t_0}$. Let $\{e_\alpha\}_{\alpha=1}^{n-1}$ be a local orthonormal frame of $\partial\Sigma_t\subset\Sigma_t$ so that $\{\mu\}\cup\{e_\alpha\}_\alpha$ forms an orthonormal frame of $\Sigma_t$. We have either $\xi=\mu$ or $\xi=e_\alpha$ for some $\alpha$ on $\partial\Sigma_{t_0}$. By Proposition \ref{boundary-h-property}, we have:
	\begin{enumerate}[(i)]
		\item If $\xi=e_\alpha$, then $h_{\mu\mu}\geq h_{\a\a}=\e> 0$ and
		\begin{align*}
		\left(\nabla_\mu S_{ij}\right)\xi^i\xi^j=&\nabla_\mu h_{\alpha\alpha}=\tilde{h}_{\alpha\alpha}\left(h_{\mu\mu}-h_{\alpha\alpha}\right)\\
=&\left(\frac{1}{\sin\theta}+\cot\theta h_{\alpha\alpha}\right)(h_{\mu\mu}-\e)\geq 0.
		\end{align*}
		\item If $\xi=\mu$, then $h_{\a\a}\geq h_{\mu\mu}=\e> 0$ for all $\a=1,\cdots,n-1$. By $\nabla_\mu H=0$, we get
		\begin{align*}
		0=&\nabla_\mu h_{\mu\mu}+\sum_{\alpha}\nabla_\mu h_{\alpha\alpha}=\nabla_\mu h_{\mu\mu}+\sum_{\alpha}\left(h_{\mu\mu}-h_{\alpha\alpha}\right)\tilde{h}_{\alpha\alpha}\\
		=&\nabla_\mu h_{\mu\mu}+\sum_{\alpha}\left(\e-h_{\alpha\alpha}\right)\left(\frac{1}{\sin\theta}+\cot\theta h_{\alpha\alpha}\right),
		\end{align*}
		and hence
		\begin{equation*}
		\left(\nabla_\mu S_{ij}\right)\xi^i\xi^j=\nabla_\mu h_{\mu\mu}=\sum_{\alpha}(h_{\alpha\alpha}-\e)\left(\frac{1}{\sin\theta}+\cot\theta h_{\alpha\alpha}\right)\geq 0.
		\end{equation*}
	\end{enumerate}
Therefore, in both cases, the boundary condition \eqref{eq-bou} is satisfied.

We next check the condition \eqref{conditon-MP} whenever $S_{ij}\geq 0$ and $S_{ij}\xi^i=0$ at a point $(p_0,t_0)$.  We choose an orthonormal frame $\{e_1,\cdots,e_n\}$ around $p_0$ such that $h_{ij}$ is diagonal at $(p_0,t_0)$ with principal curvatures $\kappa_1\leq \kappa_2\leq \cdots\leq \kappa_n$ in the increasing order and $e_i$ corresponds to the principal direction of $\kappa_i$, and the null vector $\xi=e_1$.  The lower order terms in \eqref{s5:evol-S_ij-strict-convex} involving $S_{ij}$ and $(S^2)_{ij}$ can be ignored due to the null vector condition. To apply Theorem \ref{s1:thm-max principle}, we need to check that
	\begin{align}\label{s4:tilde-Q_1}
	Q_1:=&2\nabla_1 H \Big(\langle e_1,e\rangle \langle x,\nu \rangle-\langle e_1,x\rangle \langle e,\nu\rangle+\e \langle X_e,e_1\rangle\Big) \nonumber\\
    &-(H-n\e)\langle e,\nu\rangle +\e\Big((|A|^2-\e H)\langle X_e,\nu\rangle+(H-n\e)\langle x,e\rangle\Big)\nonumber\\
	&+2\langle X_e,\nu\rangle\sup_{\G}g^{k\ell}(2\G_k^p \nabla_\ell S_{1p}-\G_k^p\G_\ell^q S_{pq})\geq 0.
	\end{align}
There are two cases:
     \begin{enumerate}
       \item[(1)] $\e=\kappa_1=\cdots =\kappa_n$ at $(p_0,t_0)$;
       \item[(2)] there exists an $1\leq m<n$ such that $\e=\kappa_1=\cdots=\kappa_m<\kappa_{m+1}\leq \cdots\leq \kappa_n$  at $(p_0,t_0)$.
     \end{enumerate}
 We consider the two cases separately.

For \textbf{case (1)}, we have $H=n\e$ and $|A|^2-\e H=0$. So the second line of \eqref{s4:tilde-Q_1} vanishes. Moreover, by Lemma \ref{s3:tech-lemma}, $  \nabla_kS_{\ell\ell}=\nabla_kh_{\ell\ell}=0$ for all $1\leq k,\ell\leq n$. This implies that $\nabla_1H=0$. Then $Q_1\geq 0$ by simply choosing $\Gamma_k^p=0$ for all $k,p=1,\cdots,n$.

For \textbf{case (2)}, we have $S_{11}=\cdots S_{mm}=\k_1-\e=0$ at $(x_0,t_0)$ and $S_{\ell\ell}>0$ for $\ell=m+1,\cdots,n$. Then
	\begin{align*}
	|A|^2-\e H>0, \quad H-n\e>0.
	\end{align*}
Again by Lemma \ref{s3:tech-lemma},  the minimality of $S_{11},\cdots, S_{mm}$ implies that $\nabla_k S_{\ell\ell}=0$ for all $k=1,\cdots,n$ and $\ell=1,\cdots,m$. Then the last term in \eqref{s4:tilde-Q_1} can be explicitly computed as follows
    \begin{align}\label{s4.Gam}
    &\sup_{\G}g^{k\ell}(2\G_k^p\nabla_\ell S_{1p}-\G_k^p\G_\ell^q S_{pq})\nonumber\\
    =&\sup_{\G}\sum_{k=1}^n\(\sum_{\ell=1}^{m}2\G^\ell_k\nabla_k S_{1\ell}+\sum_{\ell=m+1}^{n}(2\G^\ell_k\nabla_k S_{1\ell}-(\G_k^\ell)^2S_{\ell\ell})\)\nonumber\\
      =&\sup_{\G}\sum_{k=1}^n\(\sum_{\ell=1}^{m}2\G^\ell_k\nabla_k h_{1\ell}+\sum_{\ell=m+1}^{n}\(\frac{|\nabla_k S_{1\ell}|^2}{S_{\ell\ell}}-\(\G_k^\ell-\frac{\nabla_k S_{1\ell}}{S_{\ell\ell}}\)^2S_{\ell\ell}\)\)\nonumber\\
   \geq &\sum_{k=1}^n\sum_{\ell=m+1}^{n}\frac{|\nabla_k S_{1\ell}|^2}{S_{\ell\ell}}\nonumber\\
   \geq & \sum_{\ell=m+1}^{n}\frac{|\nabla_1 h_{\ell\ell}|^2}{\kappa_{\ell}-\e},
    \end{align}
    where in the first inequality we have chosen $\G_k^\ell=\frac{\nabla_k S_{1\ell}}{S_{\ell\ell}}$ for $\ell=m+1,\cdots,n$ and $\G_k^\ell=0$ for $\ell=1,\cdots,m$, and in the last inequality we used the Codazzi equation $\nabla_kh_{1\ell}=\nabla_1h_{k\ell}$ and discarded the terms for $k\neq \ell$.

    For the first term of \eqref{s4:tilde-Q_1}, we calculate that
	\begin{align}\label{s4:Xe}
	&\Big\langle e_1,e\langle x,\nu\rangle-x\langle e,\nu\rangle+\e X_e\Big\rangle^2\nonumber\\
\leq &|e\langle x,\nu\rangle-x\langle e,\nu\rangle+\e X_e|^2-\langle \nu, e\langle x,\nu\rangle-x\langle e,\nu\rangle+\e X_e\rangle^2 \nonumber\\
=&\langle x,\nu\rangle^2+|x|^2\langle e,\nu\rangle^2-2\langle x,\nu\rangle\langle e,\nu\rangle \langle x,e\rangle \nonumber\\
&+2\e\langle X_e,e\langle x,\nu\rangle-x\langle e,\nu\rangle \rangle+\e^2(|X_e|^2-\langle X_e,\nu\rangle^2)\nonumber\\
	= &-2\langle X_e,\nu\rangle\langle e,\nu\rangle+\langle x,\nu\rangle^2-\langle e,\nu\rangle^2 \nonumber\\
	&	+2\e\langle X_e,e\langle x,\nu\rangle-x\langle e,\nu\rangle \rangle+\e^2(|X_e|^2-\langle X_e,\nu\rangle^2).
	\end{align}
By \eqref{s2.est2} and Proposition \ref{s2:prop-estimate-convex}, we have
	\begin{align}\label{s4:key-ineq}
		\langle x,\nu\rangle^2-\langle e,\nu\rangle^2=\langle x+e,\nu\rangle \langle x-e,\nu\rangle\leq -\d_1\d_4,
	\end{align}
where $\delta_1, \delta_4$ are constants in Proposition \ref{s2:prop-estimate-convex}. 	 Since
	\begin{align}\label{s4:key-ineq-2}
	\langle X_e,\nu\rangle \leq |X_e|\leq 2, \quad
	|e\langle x,\nu\rangle-x\langle e,\nu\rangle|\leq 2,
	\end{align}
it follows from \eqref{s4:Xe}, \eqref{s4:key-ineq} and \eqref{s4:key-ineq-2} that
	\begin{align*}
	\langle e_1,e\langle x,\nu\rangle-x\langle e,\nu\rangle+\e X_e\rangle^2	\leq  -2\langle X_e,\nu\rangle\langle e,\nu\rangle -\d_1\d_4+4\e^2+8\e.
	\end{align*}
	Note that $\delta_1, \delta_4$ depend only on $R_1, R_2$ and $\theta$. By choosing $\e>0$ small enough such that $4\e^2+8\e \leq \d_1\d_4$, we have
\begin{equation}\label{s4.thm2-pf1}
 \langle e_1,e\langle x,\nu\rangle-x\langle e,\nu\rangle+\e X_e\rangle^2	\leq  -2\langle X_e,\nu\rangle\langle e,\nu\rangle.
\end{equation}

Then using \eqref{s4.Gam} and \eqref{s4.thm2-pf1},  we obtain the following estimate on $Q_1$:
	\begin{align}\label{s4:tilde-Q_1b}
	Q_1\geq & -2|\sum_{\ell=m+1}^{n}\nabla_1h_{\ell\ell}|\sqrt{-2\langle X_e,\nu\rangle\langle e,\nu\rangle}-\sum_{\ell=m+1}^{n}(\kappa_{\ell}-\e)\langle e,\nu\rangle\nonumber\\
&\quad +2\langle X_e,\nu\rangle \sum_{\ell=m+1}^{n}\frac{|\nabla_1 h_{\ell\ell}|^2}{\kappa_{\ell}-\e}\nonumber\\
\geq &\sum_{\ell=m+1}^{n}\left(\sqrt{\kappa_{\ell}-\e}\sqrt{-\langle e,\nu\rangle }-\sqrt{2\langle X_e,\nu\rangle}\frac{|\nabla_1h_{\ell\ell}|}{\sqrt{\kappa_{\ell}-\e}}\right)^2	\nonumber\\
\geq &0.
 \end{align}

In summary, in both two cases, we have ${Q}_1\geq 0$ and thus the condition \eqref{conditon-MP} in Theorem \ref{s1:thm-max principle} is satisfied. Therefore, we conclude that $S_{ij}\geq 0$ is preserved for $t>0$. This implies that $h_{ij}\geq \e g_{ij}>0$ for $t\in [0,T)$.
\end{proof}

\section{Proofs of Theorem \ref{thm-con} and Theorem \ref{co-af}}\label{sec:6}
In this section, we prove the long time existence and smooth convergence of the flow \eqref{Guan-Li-flow} for strictly convex hypersurfaces in the unit ball with $\theta$-capillary boundary. As an application, we prove a family of new inequalities for quermassintegrals $W_{k,\theta}(\widehat{\Sigma})$.

\subsection{Long-time existence}
Let $\Sigma\subset\bar{\mathbb{B}}^{n+1}(n\geq 2)$ be a smooth properly embedded, strictly convex hypersurface with $\theta$-capillary boundary ($\theta\in (0,\frac{\pi}{2}]$), given by an embedding: $x:\bar{\mathbb{B}}^n\to\Sigma\subset\bar{\mathbb{B}}^{n+1}$. By Theorem \ref{s4:thm-strict-convex}, the solution $\Sigma_t$ of the flow \eqref{s3:GL-flow} starting from $\Sigma$ remains to be strictly convex for $t>0$. Proposition \ref{s2:prop-estimate-convex} implies that $\langle X_e,\nu\rangle \geq \d_2>0$ and hence $\Sigma_t$ is star-shaped for $t\geq 0$. Then we can reduce the flow \eqref{s3:GL-flow} to a scalar parabolic equation with oblique boundary condition as in \cite{WW2020,Wang-Xia2019}

We briefly review the transformation of the flow \eqref{Guan-Li-flow} to a scalar parabolic equation on the hemisphere $\bar{\mathbb{S}}_+^n$, and refer the readers to \cite{WW2020,Wang-Xia2019} for more details. Assume that $e$ is the $(n+1)$-th coordinate vector $e_{n+1}$.  Consider the following M\"obius transformation:
\begin{align}\label{s6:conform-transform}
\varphi: ~\-{\mathbb B}^{n+1} \quad &\longrightarrow \quad \-{\mathbb R}_{+}^{n+1},\nonumber\\
(x',x_{n+1})  &\longmapsto  \frac{2(x',0)+(1-|x'|^2-x_{n+1}^2)e}{|x'|^2+(x_{n+1}-1)^2}:=(y',y_{n+1})=\~y,
\end{align}
where $x=(x',x_{n+1})$ with $x'=(x_1,\cdots,x_n)\in \mathbb R^{n}$ and $x_{n+1}=\langle x,e\rangle$. Then $\varphi$ maps $\mathbb S^n=\partial \mathbb B^{n+1}$ to $\partial \mathbb R_{+}^{n+1}: =\{(y',y_{n+1})\in \mathbb R^{n+1}:y_{n+1}=0\}$ and
\begin{align*}
\varphi^\ast(\d_{\mathbb R^{n+1}_{+}})=\frac{4}{(|x'|^2+(x_{n+1}-1)^2)^2}\d_{\mathbb B^{n+1}}:=e^{-2w}\d_{\mathbb B^{n+1}},
\end{align*}
which implies that $\varphi$ is a conformal transformation from $(\-{\mathbb B}^{n+1},\d_{\mathbb B^{n+1}})$ to $(\-{\mathbb R}_{+}^{n+1},\d_{\mathbb R_+^{n+1}})$. Then a properly embedded hypersurface $\Sigma_t=x(\-{\mathbb B}^n,t)$ in $(\-{\mathbb B}^{n+1},\d_{\mathbb B^{n+1}})$ can be identified with $\widetilde{\Sigma}_t:=\~y(\-{\mathbb B}^n,t)$ in $(\-{\mathbb R}_{+}^{n+1},(\varphi^{-1})^\ast\d_{\mathbb R_+^{n+1}})$, where $\~y=\varphi\circ x$.

Note that $X_e$ is a conformal vector field such that $\varphi_{\ast}(X_e)=-\~y$. For a hypersurface $\Sigma \subset \-{\mathbb B}^{n+1}$ with capillary boundary $\partial \Sigma\subset \mathbb S^n$, one has
\begin{align*}
e^{-2w}\langle X_e,\nu\rangle=\langle \varphi_\ast(X_e),\varphi_\ast(\nu)\rangle=|\varphi_{\ast}(\nu)|\langle \~y,\~\nu\rangle,
\end{align*}
where $|\varphi_\ast(\nu)|=\frac{1}2(|y'|^2+(y_{n+1}+1)^2)$ and $\~\nu:=-\frac{\varphi_\ast(\nu)}{|\varphi_\ast(\nu)|}$. Hence, the hypersurface $\varphi(\Sigma)$ is star-shaped in $\mathbb R_{+}^{n+1}$ with respect to the origin, i.e., $\langle \~y,\~\nu\rangle>0$ on $\varphi(\Sigma)$,  if and only if $\langle X_e,\nu\rangle>0$  on $\Sigma$. In particular, since $\langle X_e,\nu\rangle>0$ on $\Sigma_t$ by Proposition \ref{s2:prop-estimate-convex}, the hypersurface $\widetilde{\Sigma}_t$ in $(\-{\mathbb R}_{+}^{n+1},(\varphi^{-1})^\ast\d_{\mathbb R_+^{n+1}})$ can be written as a radial graph over $\-{\mathbb S}^{n}_{+}$.

In $\mathbb R_{+}^{n+1}$, we use the polar coordinate $\tilde{y}=(\rho,z)\in [0,\infty)\times \bar{\mathbb{S}}_+^n$, where $\rho$ is the distance from $\tilde{y}$ to the origin and we write $z=(\b,\xi)\in [0,\frac{\pi}{2}]\times \mathbb S^{n-1}$ for the spherical polar coordinate of $z\in \mathbb S^{n}$. Then
\begin{align}\label{s6:rho-def}
\left\{\begin{aligned}
\rho^2=&|y'|^2+y_{n+1}^2,\\
y_{n+1}=&\rho \cos \b,\quad |y'|=\rho \sin \b.
\end{aligned}   \right.
\end{align}
Let $u:=\log \rho$ and $v:=\sqrt{1+|\nabla^{\mathbb S}u|^2}$, where $\nabla^{\mathbb S}$ is the Levi-Civita connection on $\mathbb S^{n}_+$ with respect to the round metric $\s$. We have
\begin{align}\label{s6:<X,nu>-expression}
\langle X_e,\nu \rangle=e^w\frac{\rho}{v}=\frac{2\rho}{\sqrt{1+|\nabla^{\mathbb S}u|^2}(\rho^2+2\rho\cos \b+1)}.
\end{align}

Furthermore, up to a time-dependent tangential diffeomorphism, one can rewrite the flow \eqref{Guan-Li-flow} equivalently as the following scalar parabolic equation on $\-{\mathbb S}_{+}^{n}$:
\begin{equation}\label{Guan-Li-scalar-flow}
\left\{\begin{aligned}
\frac{\partial}{\partial t} u&=F(\nabla^2_{\mathbb S} u, \nabla^\mathbb S u, \rho, \b) \qquad &\text{in}\quad \mathbb S_{+}^n \times[0,T),\\
\nabla^{\mathbb S}_{\partial_\b}u&= -\cos\t \sqrt{1+|\nabla^{\mathbb S}u|^2}\qquad &\text{on}\quad \partial\mathbb S_+^{n} \times[0,T),\\
u(\cdot,0)&=u_0(\cdot) \qquad &\text{on} \quad \mathbb S_{+}^n,
\end{aligned}\right.
\end{equation}
where $u_0=\log\rho_{0}$, $\rho_0$ is radial function of the initial hypersurface $\varphi(\Sigma_0)$, $\partial_\b$ is the unit outward normal of $\partial \mathbb S_+^n$ on $\-{\mathbb S}_+^{n}$ and
\begin{align*}
F(\nabla^2_{\mathbb S} u, \nabla^\mathbb S u, \rho, \b):=&\operatorname{div}_{\mathbb S^{n}_{+}}\(\frac{\nabla^{\mathbb S} u}{\rho e^w v}\)-\frac{n+1}{v}\s\(\nabla^{\mathbb S} u,\nabla^{\mathbb S}(\frac{1}{\rho e^w})\)\\
&-\frac{n\cos\t}{2}\cdot \frac{\rho^2-1}{\rho}\sin\b \s(\nabla^{\mathbb S} u,\partial_\b)\\
                                 &+\frac{n\cos \t}{2}\cdot \frac{\rho^2 \cos\b+2\rho+\cos\b}{\rho}.
\end{align*}
As $\theta\in (0,\frac{\pi}{2}]$, we have $\cos\theta \in [0,1)$. Then we have a uniformly oblique boundary condition in \eqref{Guan-Li-scalar-flow}.

The barrier estimate in Lemma \ref{s3:lem-barrier} and the convexity of $\Sigma_t$ imply the uniform $C^1$ estimate of $\Sigma_t$. In fact, under the transformation \eqref{s6:conform-transform}, using \eqref{s6:rho-def}  we have
\begin{align*}
\rho^2=&\frac{|x'|^4+(1-x_{n+1}^2)^2+2(1+x_{n+1}^2)|x'|^2}{|x'|^4+(1-x_{n+1})^4+2(1-x_{n+1})^2|x'|^2}\\
      =&1+\frac{4x_{n+1}}{|x'|^2+(1-x_{n+1})^2}.
\end{align*}
By Proposition \ref{s2:prop-estimate-convex}, we have $x_{n+1}=\langle x,e\rangle \geq \cos\t+\d_0$ and $\langle x-e,\nu\rangle \geq \d_1>0$. Then
\begin{align*}
\rho^2 \geq & 1+2(\cos\t+\d_0) \geq 1+2\d_0, \\
\rho^2 \leq & 1+\frac{4}{|x-e|^2} \leq 1+\frac{4}{|\langle x-e,\nu\rangle|^2}\leq 1+\frac{4}{\d_1^2}.
\end{align*}
Since $u=\log \rho$, we obtain the uniform $C^0$ estimate of $u$, i.e.,
\begin{equation*}
  0<c_1\leq u\leq c_2<\infty
\end{equation*}
for some constants $c_1,c_2$ depending on $\Sigma_0$. For the $C^1$ estimate, using $\langle X_e,\nu\rangle\geq \d_2>0$ on $\Sigma_t$ and \eqref{s6:<X,nu>-expression} we obtain that
\begin{align*}
|\nabla^{\mathbb S}u|\leq &\sqrt{1+|\nabla^{\mathbb S}u|^2}\\
 \leq &\frac{2\rho}{\d_2\sqrt{\rho^2+2\rho\cos\b+1}}\\
 \leq &\frac{2}{\d_2}\sqrt{1+\frac{4}{\d_1^2}}=:c_3,
\end{align*}
where $c_3$ is a positive constant depending on $\Sigma_0$. It follows that the scalar flow equation \eqref{Guan-Li-scalar-flow} is a quasilinear parabolic equation with uniform oblique boundary condition. The higher order estimates of $u$ follows from the classical parabolic theory for quasi-linear parabolic equations \cite[Theorem 7.4]{Ural1991} (see also \cite[\S 13]{Lieb96}). As a consequence, the solution $u$ of \eqref{Guan-Li-scalar-flow} exists for all time $t\in [0,\infty)$. Equivalently, we have
\begin{prop}\label{s5:long-time-existence}
Let $\Sigma\subset\bar{\mathbb{B}}^{n+1}(n\geq 2)$ be a properly embedded, strictly convex hypersurface with $\theta$-capillary boundary $\partial\Sigma\subset \mathbb{S}^n$ ($\theta\in (0,\frac{\pi}{2}]$), given by an embedding: $x:\bar{\mathbb{B}}^n\to\Sigma\subset\bar{\mathbb{B}}^{n+1}$. Then the solution $\Sigma_t$ of the flow \eqref{Guan-Li-flow} starting from $\Sigma$ has uniform $C^k$ regularity estimates for all $k\geq 0$ and exists for all time $t\in [0,\infty)$.
\end{prop}

\subsection{Convergence to a spherical cap}
To prove the smooth convergence to the spherical cap, we explore the monotonicity of the flow \eqref{s3:GL-flow}. By the case $k=1$ of \eqref{s5:monotonicity}, we have
	\begin{align}\label{eq-W_1}
	\frac{d}{dt}{W_{1,\theta}(\widehat{\Sigma_t})}=\frac{n^2}{n+1}\int_{\Sigma_t}(H_2-H_1^2)\langle X_e,\nu\rangle dA_t \leq 0.
	\end{align}
The $C^0$ estimate of $\Sigma_t$ implies that ${W_{1,\theta}(\widehat{\Sigma_t})}$ is uniformly bounded from below. Then the long-time existence and uniform $C^{\infty}$-estimates imply that
	\begin{equation*}
	\int_{\Sigma_t}{(H_2-H_1^2)\langle X_e,\nu\rangle dA_t}\to 0,\, \quad \text{as $t\to\infty$}.
	\end{equation*}
Since $\langle X_e,\nu\rangle\geq \delta_2>0$ along the flow, any limit of $\Sigma_t$ is totally umbilical and hence must be a spherical cap.

The limit spherical cap is uniquely determined and is independent of the subsequence of times. In fact, for any sequence of times $t_j\to\infty$ such that $\Sigma_{t_j}$ converges smoothly to a $\theta$-capillary spherical cap $C_{\theta,r_{\infty}}(e_{\infty})$, the radius $r_\infty$ is uniquely determined by the fact $W_{0,\t}(\widehat{C_{\theta,r_{\infty}}}(e_{\infty}))=W_{0,\t}(\widehat{\Sigma})$. Furthermore, we must have $e_\infty=e$, by using a similar argument as in \cite{Scheuer-Wang-Xia2018,Weng-Xia2022} for the locally constrained inverse curvature flow \eqref{inverse-harmonic-mean-curvature-flow}. The idea is simply that any spherical cap $C_{\theta,r_{\infty}}(e_{\infty})$ with the center $e_\infty\neq e$ is not stationary and must evolve towards $e$ along the flow \eqref{s3:GL-flow}. We refer the readers to \cite{Scheuer-Wang-Xia2018,Weng-Xia2022} for the similar justification of this fact along the flow \eqref{inverse-harmonic-mean-curvature-flow}.

Therefore, we conclude that along the flow \eqref{s3:GL-flow}, the solution $\Sigma_t$ converges smoothly to the unique spherical cap $C_{\t,r_\infty}(e)$. This completes the proof of Theorem \ref{thm-con}.

\subsection{Proof of Theorem \ref{co-af}}
Firstly, we assume that $\Sigma$ is strictly convex, then by Proposition \ref{s2:prop-estimate-convex}, there exists a constant vector $e\in \mathbb S^n$ such that $\langle X_e,\nu\rangle>0$ holds everywhere on $\Sigma$. Start the flow \eqref{s3:GL-flow} from $\Sigma$. By Theorem \ref{thm-con}, the solution $\Sigma_t$ is strictly convex and converges smoothly to a spherical cap $C_{\t,r_\infty}(e)$ as $t\ra \infty$. By the monotonicity in Proposition \ref{prop-mono}, we have
\begin{align*}
 W_{0,\theta}(\widehat{\Sigma})=&W_{0,\theta}(\widehat{C_{\theta,r_{\infty}}}(e))=f_0(r_\infty) \\
 W_{k,\theta}(\widehat{\Sigma})\geq & W_{k,\theta}(\widehat{C_{\theta,r_{\infty}}}(e))=f_k(r_\infty).
\end{align*}
Since $f_k$ is strictly increasing, we can rewrite the above two equations as
	\begin{align}\label{s5:final-ineq}
	W_{k,\theta}(\widehat{\Sigma}) \geq f_k\circ f_0^{-1}\left(W_{0,\theta}(\widehat{\Sigma})\right).
	\end{align}
If equality holds in \eqref{s5:final-ineq}, then all $\Sigma_t$ are spherical caps with $\t$-capillary boundary for $t>0$ and in particular the initial hypersurface $\Sigma$ is also a spherical cap with $\t$-capillary boundary.

When $\Sigma$ is convex, we may assume that it is not a flat ball, otherwise the statement is trivially true.  By Theorem \ref{thm.app}, we can approximate $\Sigma$ by a family of strictly convex hypersurfaces $\Sigma_{\varepsilon}$ as $\varepsilon\to 0$. Then the inequality \eqref{eq-AF} follows by approximation. The equality characterization can be proved by using an argument of \cite{GL09} and noting the fact that $\Sigma$ has an interior strictly convex point.

\subsection{Further question}

To establish the Alexandrov-Fenchel type inequalities comparing quermassintegrals $W_{k,\theta}$ and $W_{\ell,\theta}$ for $k>\ell$, it's also natural to study the following fully nonlinear locally constrained curvature flow
\begin{equation}\label{s5.GLflow}
\left\{\begin{aligned}
\left(\partial_t x\right)^{\bot}&=\Big(\langle x+\cos\theta\nu,e\rangle-\frac{H_k}{H_{k-1}}\langle X_e,\nu\rangle\Big)\nu &\text{in}\quad \bar{\mathbb{B}}^n \times[0,T),\\
\langle\bar{N}\circ x,\nu\rangle&=-\cos\theta  &\text{on}\quad \partial\bar{\mathbb{B}}^n \times[0,T),\\
x(\cdot,0)&=x_0(\cdot)  &\text{on} \quad \bar{\mathbb{B}}^n,
\end{aligned}\right.
\end{equation}
for $k=2,\cdots, n$. Along the flow \eqref{s5.GLflow}, $W_{k-1,\theta}$ is preserved while $W_{k,\theta}$ is decreasing. Using the maximum principle, the two-sided positive bounds on the curvature function $H_k/H_{k-1}$ can be proved. Furthermore, we can prove that the strict convexity is preserved along the flow \eqref{s5.GLflow} by using the similar argument as in \S \ref{sec:5}. The curvature upper bound is not yet available. However, we still expect that this is true and then this would give the smooth convergence of the flow to a spherical cap.

\appendix
\section{Approximation result}
In this appendix, we prove that a convex hypersurface $\Sigma$ with $\theta$-capillary boundary in the unit ball can be approximated by a sequence of strictly convex hypersurfaces with $\theta$-capillary boundary in the unit ball in the $C^{2,\alpha}$ sense. The free boundary case (i.e. $\theta=\frac{\pi}{2}$) has been treated by Lambert and Scheuer in \cite{Lambert-Scheuer2017}. We adapt their idea and include a proof for the capillary case.

Firstly, we show the following lemma which could be considered as a natural generalization of \cite[Lemma 3.1]{Lambert-Scheuer2017}.
\begin{lem}\label{s2:prop-weak-convex-rigidity}
	Let $\Sigma\subset\bar{\mathbb{B}}^{n+1}$ be a convex hypersurface with $\theta$-capillary boundary for $\theta\in(0,\frac{\pi}{2}]$. Then there exists a point $e\in \operatorname{int}(\widehat{\partial\Sigma})$, such that $\widehat{\partial \Sigma}\subset \widehat{\partial C_{\theta,\infty}}(e)$. Moreover, either $\partial\Sigma=\partial C_{\theta,\infty}(e)$,  or $\partial\Sigma\subset \operatorname{int}(\widehat{\partial C_{\theta,\infty}}(e))$ and in this case there exists an interior point in $\Sigma$ such that $h_{ij}>0$ at this point.
\end{lem}	
\begin{proof}
The case $\theta=\frac{\pi}{2}$ is treated in \cite[Lemma 3.1]{Lambert-Scheuer2017}. In the following, we assume $\theta\in (0,\frac{\pi}{2})$.

\textbf{\underline{Step 1.}} We first show that there exists $e\in \operatorname{int}(\widehat{\partial\Sigma})$ such that $\widehat{\partial \Sigma} \subset \widehat{\partial C_{\theta,\infty}}(e)$.

By Proposition \ref{boundary-h-property} and $\Sigma$ is convex, one has $\widehat{h}_{\a\b} \geq \cot \t \d_{\a\b}>0$ on $\partial\Sigma$. By \cite[Theorem 1.2]{NAndo}, we know that $\widehat{\partial \Sigma} \subset \bigcap_{q\in \partial \Sigma}\widehat{\mathcal{S}_{q,\t}}$, where $\mathcal{S}_{q,\t}$ is the circumscribed sphere of $\partial \Sigma$ in $\mathbb S^n$ with radius $\t$ such that it touches $\partial\Sigma$ tangentially at $q$. On the other hand, by the convexity of $\partial \Sigma$, we have $\widehat{\partial\Sigma}=\bigcap_{q\in \partial\Sigma}\widehat{\mathcal{S}_{q,\frac{\pi}{2}}}$ which contains  $\bigcap_{q\in \partial \Sigma}\widehat{\mathcal{S}_{q,\t}}$. Therefore, we have $\widehat{\partial \Sigma}=\bigcap_{q\in \partial \Sigma}\widehat{\mathcal{S}_{q,\t}}$.
	
	Let $\-B_r(e)\subset \mathbb S^n$ be the closed geodesic ball with the smallest radius $r$ and center $e$ such that $\widehat{\partial \Sigma}\subset \-B_r(e)$. It is obvious that $r\leq \t$. We show that $e\in \operatorname{int}(\widehat{\partial\Sigma})$. Let $C=S_r(e)\cap \partial \Sigma$, where $S_r(e)=\partial B_r(e)$. The minimality of $r$ implies that $C$ must support $\-B_r(e)$: for each hyperplane of $\mathbb{S}^n$ which passes through $e$, $S_r(e)$ is cutted into two hemispheres $S_r(e)=S_1\cup S_2$ and must satisfes $\bar{S_i}\cap C\neq\emptyset,\ i=1,2.$ This leads to the following two cases.
	
	{ Case 1}: If two points $q_1,q_2$ of $C$ are antipodal on $S_r(e)$, then $e$ is the midpoint of the minimal geodesic $\g$ joining $q_1$ and $q_2$. Then $e\in \g \subset \widehat{\partial \Sigma}$ by convexity. If $e\in \partial \Sigma$, then $\g \subset \widehat{\mathcal{S}_{e,\t}}$ implies that $\t=\frac{\pi}{2}$ and $\g\subset \partial \Sigma$ by the convexity of $\partial \Sigma$, contradicting with $\theta<\frac{\pi}2$. Therefore, we have $e\in \operatorname{int}(\widehat{\partial \Sigma})$.
	
	{ Case 2}: If $C$ contains no pair of antipodal points of $S_r(e)$, since $C$ support $B_r(e)$, then $C$ contains at least three points $q_1,q_2,q_3$ which would not lie simultaneously in any hemisphere of $S_r(e)$. Then the convex hull of $q_1,q_2,q_3$, which is a spherical triangle, would not be contained in any half ball of $B_r(e)$. Then we must have $e\in \operatorname{int}(\operatorname{conv}(q_1,q_2,q_3)) \subset  \operatorname{int}(\widehat{\partial \Sigma})$.
	
	\textbf{\underline{Step 2.}} We next show that either $\partial \Sigma=\partial C_{\theta,\infty}(e)$ or $\partial\Sigma\subset \operatorname{int}(\widehat{\partial C_{\theta,\infty}}(e))$.

For any $q_1,q_2 \in \partial\Sigma$, we take $p_1,p_2 \in \mathbb S^n$ such that $\widehat{\mathcal{S}_{q_i,\t}}=\-B_{\t}(p_i)$ for $i=1,2$. If $p_1 \neq p_2$, then
	$$
	\widehat{\partial \Sigma}\subset \-B_{\t}(p_1)\cap \-B_{\t}(p_2).
	$$
	Let $o$ be the midpoint of the geodesic between $p_1$ and $p_2$. For any $z\in \partial(\-B_\t(p_1)\cap \-B_\t(p_2))$, we draw the geodesics from the point $z$ to the points $o$, $p_1$ and $p_2$, respectively. Denote by $\angle zop_i$ the angle of the spherical triangle $\D zop_i$ at the vertex $o$ for $i=1,2$, respectively. Without loss of generality, we assume that $z\in \partial B_{\t}(p_1)$, then $\angle zop_1\geq \pi/2$, see Figure \ref{fig2}.
	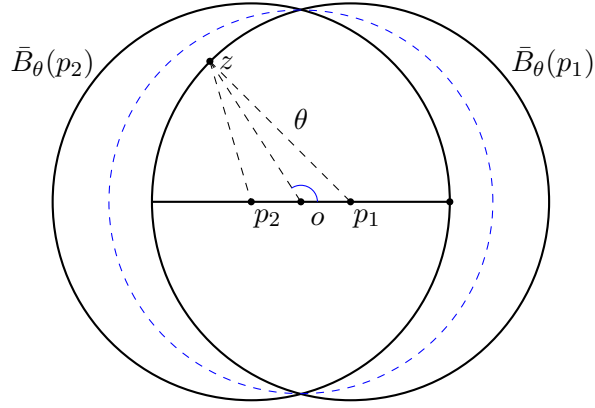
\begin{figure}[htbp]
		\begin{tikzpicture}[scale=1.1]
		\draw[thick] (0.6,0) arc (0:360:2.4);
		\draw[thick] (-0.6,0) arc (0:360:2.4);
		\draw[thick] (-0.6,0)--(-4.2,0);
		\draw[dashed,blue] (-2.4,0) circle(2.32);
		\filldraw[thick](-0.6,0) circle(0.03);
		\filldraw[thick](-1.8,0) circle(0.03);
		\node at (-1.6,-0.2) {$p_1$};
		\filldraw[thick](-3,0) circle(0.03);
		\node at (-2.8,-0.2) {$p_2$};
		\node at (-2.4,1) {$\t$};
		\filldraw[thick](-2.4,0) circle(0.03);
		\node at (-2.2,-0.2) {$o$};
		\draw[thin,blue] (-2.2,-0) arc (0:125:0.2);
		
		\filldraw[thick](-3.5,1.7) circle(0.03);
		\node at (-3.3,1.68) {$z$};
		\node at (-5.4,1.68) {$\bar{B}_{\t}(p_2)$};
		\node at (0.65,1.68) {$\bar{B}_{\t}(p_1)$};
		
		\draw[dashed] (-3.5,1.7)--(-2.4,0);
		\draw[dashed] (-3.5,1.7)--(-1.8,0);
		\draw[dashed] (-3.5,1.7)--(-3,0);
		
		\end{tikzpicture}
		\caption{The intersection of $\bar{B}_\t(p_1)$ and $\bar{B}_\t(p_2)$}\label{fig2}
	\end{figure}
	
	Let $d^{\mathbb S}$ be the distance function on $\mathbb S^n$. By the Law of Cosines on the sphere \cite{Whittlesey}, we have
	\begin{align*}
	\cos d^{\mathbb S}(z,o)\cos d^{\mathbb S}(o,p_1)+\cos\angle zo p_1 \sin d^{\mathbb S}(z,o)\sin d^{\mathbb S}(o,p_1)=\cos d^{\mathbb S}(z,p_1).
	\end{align*}
	Then we get
	\begin{align*}
	\cos d^{\mathbb S}(z,o)>&\cos d^{\mathbb S}(z,o)\cos d^{\mathbb S}(o,p_1)\\
	\geq &\cos d^{\mathbb S}(z,p_1)=\cos \t,
	\end{align*}
	and hence $d^{\mathbb S}(z,o)<\t$. Thus, the minimality of $\-B_r(e)$ implies that either $r=\t$ or $r<\t$.
	
	In the latter case, we have $\widehat{\partial\Sigma}\subset \-B_r(e)$ with $r<\t$. Then we have $\langle x,e\rangle \geq \cos r$ and $0\geq \langle \bar{\nu},e\rangle\geq -\sin r$ on $\partial \-{\mathbb B}^n$. Using the relation \eqref{normal transform}, we obtain
	\begin{align*}
	D_\mu\langle x,e\rangle=\langle \mu,e \rangle=&\sin\t \langle \-N\circ x,e\rangle+\cos \t \langle \-\nu,e\rangle \\
	\geq &\sin\t\cos r-\cos\t\sin r\\
	=&\sin(\t-r)>0,\quad \text{on $\partial\-{\mathbb B}^n$}.
	\end{align*}
	Then the strict minimum of $\langle x,e\rangle$ is attained at some interior point in $\mathbb B^n$. By attaching a large supporting sphere to $\Sigma$ from below, we obtain an interior strictly convex point in $\mathbb B^n$.
\end{proof}

Let $\Sigma\subset\bar{\mathbb{B}}^{n+1}$ be a convex hypersurface with capillary boundary supported on $\mathbb{S}^n$ at a contact angle $\theta\in(0,\frac{\pi}{2}]$, which is given by the embedding $x_0:\bar{\mathbb{B}}^{n}\to\Sigma\subset\bar{\mathbb{B}}^{n+1}$. We consider the mean curvature flow:
\begin{equation}\label{flow-mean}
	\left\{\begin{aligned}
		\partial_t x&=-H\nu+V \qquad \text{in}\quad \bar{\mathbb{B}}^n \times[0,T),\\
		\langle\bar{N}\circ x,\nu\rangle&=-\cos\theta \qquad \text{on}\quad \partial\bar{\mathbb{B}}^n \times[0,T),\\
		x(\cdot,0)&=x_0(\cdot) \qquad \text{on} \quad \bar{\mathbb{B}}^n,
	\end{aligned}\right.
\end{equation}
and denote $\Sigma_t=x(\bar{\mathbb{B}}^n,t)$, where $V$ is the tangential component of $\partial_tx$ and  satisfies $V\big|_{\partial\Sigma_t}=-H\cot\theta\mu$.

The short time existence of the flow \eqref{flow-mean} can be obtained by a similar argument as in \cite{Stahl1996-2} by Stahl.
\begin{thm}\label{short time}
	For any $\alpha\in(0,1)$, the flow \eqref{flow-mean} admits a unique solution:
	\begin{equation*}
		x(\cdot,t)\in C^{\infty}(\bar{\mathbb{B}}^n\times(0,\delta])\cap C^{2+\alpha,1+\frac{\alpha}{2}}(\bar{\mathbb{B}}^n\times[0,\delta]),
	\end{equation*}
	where $\delta>0$ is a small constant.
\end{thm}

We prove the following approximation result:
\begin{thm}\label{thm.app}
Let $\Sigma\subset\bar{\mathbb{B}}^{n+1}$ be a convex hypersurface with capillary boundary supported on $\mathbb{S}^n$ at a contact angle $\theta\in(0,\frac{\pi}{2}]$.	Suppose that $\Sigma_t$, $t\in [0,\delta)$ is a solution to the flow \eqref{flow-mean} starting from $\Sigma$. Then either $\Sigma$ is a flat ball $C_{\theta,\infty}$, or $\Sigma_t$ is strictly convex for all time $t>0$ as long as the flow exists.
\end{thm}
\begin{proof}
By Lemma \ref{s2:prop-weak-convex-rigidity}, there exists a point $e\in \operatorname{int}(\widehat{\partial\Sigma})$ such that $\widehat{\partial \Sigma}\subset \widehat{\partial C_{\theta,\infty}}(e)$, and either $\partial\Sigma=\partial C_{\theta,\infty}(e)$ or $\partial\Sigma\subset \operatorname{int}(\widehat{\partial C_{\theta,\infty}}(e))$. 	

For the first case $\partial\Sigma=\partial C_{\theta,\infty}(e)$, by item (2) of Proposition \ref{boundary-h-property}, we have $h_{\alpha\beta}=0$ for $\alpha,\beta=1,\cdots,n-1$ on the boundary $\partial\Sigma$. If $h_{\mu\mu}>0$ on $\partial\Sigma$, then item (4) of Proposition \ref{boundary-h-property} implies that $\nabla_\mu h_{\alpha\beta}>0$ on $\partial\Sigma$ and then $h_{\alpha\beta}<0$ at some interior points sufficiently close to the boundary $\partial\Sigma$. So we must have $h_{\mu\mu}=0$ on $\partial\Sigma$.  This implies that if $\Sigma$ is not $C_{\theta,\infty}(e)$, then there exists an interior point where the second fundamental form is positive definite. For the second case $\partial\Sigma\subset \operatorname{int}(\widehat{\partial C_{\theta,\infty}}(e))$, by Lemma \ref{s2:prop-weak-convex-rigidity}, there is also a strictly convex point in the interior of $\Sigma$.
	
In the following, we assume that $\Sigma$ is not the flat ball $C_{\theta,\infty}(e)$.  Let
	\begin{equation*}
		\chi(x,t)=\min_{|\xi|=1}{h_{ij}{\xi}^i{\xi}^j}, \quad x\in \Sigma_t
	\end{equation*}
be the smallest principal curvature at the point $x\in \Sigma_t$. 	Since $h_{ij}$ is smooth, the function $\chi(x,t)$ is Lipschitz continuous in space and by a cut-off function argument, we find a smooth function $\phi_0:\bar{\mathbb{B}}^n\to\mathbb{R}$ such that $0\leq\phi_0\leq\chi(x,0)$ and there exists an interior point $y$ such that $\phi_0(y)>0$. We extend the function $\phi_0$ to $\phi:\bar{\mathbb{B}}^n\times[0,\delta')\to\mathbb{R}$ by solving a linear parabolic PDE:
	\begin{equation}\label{flow-heat}
		\left\{\begin{aligned}
			\frac{\partial}{\partial t}\phi&=\Delta\phi+\nabla_{V}{\phi}, \qquad \text{in}\quad \bar{\mathbb{B}}^n \times[0,\delta'),\\
			\nabla_{\mu}\phi&=0 \qquad \text{on}\quad \partial\bar{\mathbb{B}}^n \times[0,\delta'),\\
			\phi(\cdot,0)&={\phi}_0(\cdot) \qquad \text{on} \quad \bar{\mathbb{B}}^n,
		\end{aligned}\right.
	\end{equation}
	where $\Delta$ and $\nabla$ are Laplacian operator and Levi-Civita connection with respect to the induced metric on $\Sigma_t=x(\bar{\mathbb{B}}^n,t)$ of the flow \eqref{flow-mean}. The solution $\phi$ of \eqref{flow-heat} exists at least for a short time interval $[0,\delta')$. By the strong maximum principle for scalar functions (see \cite[Corollary 3.2]{Stahl1996-2}), we have $\phi(\cdot,t)>0$ in $\bar{\mathbb{B}}^n$ and $t\in(0,\delta')$.
	
	We take $\tau=\frac{1}{2}\min\{\delta,\delta'\}$ and consider the tensor:
	\begin{equation}
		M_{ij}=h_{ij}-\phi g_{ij}
	\end{equation}
	for time $t\in[0,\tau)$. By the construction of $\phi_0$, we see that $M_{ij}\geq 0$ at time $t=0$. We now apply the tensor maximum principle (i.e. Theorem \ref{s1:thm-max principle}) to deduce that $M_{ij}\geq 0$ is preserved along the flow \eqref{flow-mean}.
	
	By a direct computation using Proposition \ref{s3:lem-evol-general-speed} for $f=-H$, we have:
	\begin{equation}
		\frac{\partial}{\partial t} M_{ij}=\Delta M_{ij}+\nabla_{V} M_{ij}+N_{ij},
	\end{equation}
	where
	\begin{equation}
		N_{ij}=|A|^2 M_{ij}+|A|^2\phi g_{ij}-2HM_i^\ell M_{j\ell}-2H\phi M_{ij}+M_j^k\nabla_{i}{V_k}+M_i^k\nabla_{j}{V_k}.
	\end{equation}
	We easily see that whenever $M_{ij}\geq 0$ and $M_{ij}\xi^i=0$ at a point, we have
	\begin{equation}
		N_{ij}\xi^i\xi^j=|A|^2\phi\geq 0
	\end{equation}
	and thus the condition \eqref{conditon-MP} is satisfied. The boundary condition \eqref{eq-bou} can be checked similarly as in Theorem \ref{s4:thm-strict-convex}. Hence Theorem \ref{s1:thm-max principle} implies that $M_{ij}\geq 0$ is preserved along the flow \eqref{flow-mean} for time $t\in[0,\tau)$, and it follows that $\Sigma_t$ is strictly convex for time interval $(0,\tau)$.
	
	To show the strict convexity for the whole time interval $(0,\delta)$, we fix a time $t_0\in(0,\tau)$. Since $\Sigma_t$ is strictly convex at time $t_0$, then there exists a constant $\varepsilon>0$, such that $h_{ij}\geq \varepsilon g_{ij}$ holds everywhere on $\Sigma_{t_0}$. A similar procedure as above can be used to show that $\tilde{M}_{ij}=h_{ij}-\varepsilon g_{ij}\geq 0$ is preserved along the flow \eqref{flow-mean} for all time $t>t_0$ as long as the flow \eqref{flow-mean} exists, which finishes the proof.
\end{proof}


\begin{thebibliography}{10}
\bib{NAndo}{article}{
	author={Ando, Naoya},
	title={Strict convexity of hypersurfaces in spheres, II},
	journal={J. Math. Sci. Univ. Tokyo},
	volume={6},
	year={1999},
	pages={1--11},
}

\bib{And2007}{article}{
	author={Andrews, Ben},
	title={Pinching estimates and motion of hypersurfaces by curvature functions},
	journal={J. reine angew. Math.},
	volume={608},
	year={2007},
	pages={17--33},
}



\bib{BGL}{article}{
     author={Brendle, Simon},
     author={Guan, Pengfei},
     author={Li, Junfang},
     title={An inverse curvature type hypersurface flow in space forms},
     note={preprint},
}
	

\bib{ChenGLS22}{article}{
   author={Chen, Chuanqiang},
   author={Guan, Pengfei},
   author={ Li, Junfang},
   author={Scheuer, Julian},
   title={A fully-nonlinear flow and quermassintegral inequalities in the sphere},
   journal={ Pure Appl. Math. Q. },
   volume={18},
   date={2022},
   number={2},
   pages={437--461},
}


\bib{CW1970}{article}{
    author={do Carmo, M. P.},
    author={Warner, F.W.},
    title={Rigidity and convexity of hypersurfaces in spheres},
    journal={J. Differ. Geom.},
    volume={4},
    year={1970},
    number={2},
    pages={133--144},
}
	



\bib{Ger06}{book}{
    author={Gerhardt, Claus},
    title={Curvature problems},
    series={ Series in Geometry and Topology},
    volume={39},
    year={2006},
    publisher={International Press, Somerville, MA},
}

\bib{GL09}{article}{
	author={Guan, Pengfei},
	author={Li, Junfang},
	title={The quermassintegral inequalities for $k$-convex starshaped domains},
	journal={Adv. Math.},
	volume = {221},
	pages = {1725--1732},
	year = {2009},
}

\bib{GL-2015}{article}{
    author={Guan, Pengfei},
    author={Li, Junfang},
    title={A mean curvature type flow in space forms},
    journal={Int. Math. Res. Not.},
    year={2015},
    volume={13},
    pages={4716--4740},
}


\bib{GL-2021}{article}{
	author={Guan, Pengfei},
	author={Li, Junfang},
	title={Isoperimetric type inequalities and hypersurface flows},
	journal={J. Math. Study},
	year={2021},
	volume={54},
	number={1},
	pages={56--80},
}

\bib{GLW-2019}{article}{
	author={Guan, Pengfei},
	author={Li, Junfang},
	author={Wang, Mu-Tao},
	title={A volume preserving flow and the isoperimetric problem in warped product spaces},
	journal={Trans. Amer. Math. Soc.},
	year={2019},
	volume={372},
	pages={2777--2798},
}

\bib{Hamilton1982}{article}{
	author={Hamilton, Richard},
	title={Three-manifolds with positive Ricci curvature},
	journal={J. Differential Geom.},
	volume={17},
	year={1982},
	pages={255--306},
}


\bib{HLW2020}{article}{
	author={Hu, Yingxiang},
	author={Li, Haizhong},
	author={Wei, Yong},
	title={Locally constrained curvature flows and geometric inequalities in hyperbolic space},
	journal={Math. Ann.},
	volume={382},
	year={2022},
	pages={1425-–1474},
}

\bib{HWYZ2022}{article}{
	author={Hu, Yingxiang},
	author={Wei, Yong},
	author={Yang, Bo},
	author={Zhou, Tailong}
	title={A complete family of Alexandrov-Fenchel inequalities for convex capillary hypersurfaces in the half-space},
    journal={preprint},
	year={2022},
}

\bib{Ural1991}{book}{
    author={Lady\v{z}enskaja, O. A.},
    author={Solonnikov, V. A.},
	author={Ural'tseva, Nina N.},
	title={Linear and quasilinear equations of parabolic type. (Russian)},
	series={Translated from the Russian by S. Smith Translations of Mathematical Monographs},
	volume={23},
	year={1968},
    publisher={American Mathematical Society, Providence, R.I.},
}

\bib{Lambert-Scheuer2017}{article}{
	author={Lambert, Ben},
	author={Scheuer, Julian},
	title={A geometric inequality for convex free boundary hypersurfaces in the unit ball},
	journal={Proc. Amer. Math. Soc.},
	volume={145},
	year={2017},
	number={9},
	pages={4009--4020},
}

\bib{Lambert-Scheuer2021}{article}{
	author={Lambert, Ben},
	author={Scheuer, Julian},
	title={Isoperimetric problems for spacelike domains in generalized Robertson-Walker spaces},
	journal={ J. Evol. Equ.},
	volume={21},
	year={2021},
	number={1},
	pages={377--389},
}

\bib{Lieb96}{book}{
author={Lieberman, Gary M.},
title={Second order parabolic differential equations},
publisher={World Scientific Publishing Co., Inc., River Edge, NJ},
year={1996},
}

\bib{Sch21}{article}{
	author={Scheuer, Julian},
	title={The Minkowski inequality in de Sitter space},
	journal={ Pacific J. Math.},
	volume={314},
	year={2021},
	number={2},
	pages={425--449},
}


\bib{Scheuer-Wang-Xia2018}{article}{
    author={Scheuer, Julian},
    author={Wang, Guofang},
    author={Xia, Chao},
    title={Alexandrov-Fenchel inequalities for convex hypersurfaces with free boundary in a ball},
    journal={J. Differential Geom.},
    volume={120},
    year={2022},
    pages={345--373},
}

\bib{SX-2019}{article}{
	author={Scheuer, Julian},
	author={Xia, Chao},
	title={Locally constrained inverse curvature flows},
	journal={Trans. Amer. Math. Soc.},
	volume={372},
	year={2019},
	number={10},
	pages={6771--6803},
}

\bib{Sol06}{article}{
	author={Solanes, Gil },
	title={Integral geometry and the Gauss-Bonnet theorem in constant curvature spaces},
	journal={ Trans. Amer. Math. Soc.},
	volume={358},
	pages={1105--1115},
    number={3},
	year={2006},
}

\bib{Stahl1996-2}{article}{
	author={Stahl, Axel},
	title={Regularity estimates for solutions to the mean curvature flow with a Neumann boundary condition},
	journal={Calc. Var. Partial Differential Equations},
	volume={4},
	pages={385--407},
	year={1996},
}

%



\bib{WW2020}{article}{
	author={Wang, Guofang},
	author={Weng, Liangjun},
	title={A mean curvature type flow with capillary boundary in a unit ball},
	journal={Calc. Var. Partial Differential Equations}
	volume={59},
	year={2020},
	volume={59},
	number={5},
	pages={149--174},
}


%

\bib{Wang-Xia2019-2}{article}{
    author={Wang, Guofang},
    author={Xia, Chao},
    title={Uniqueness of stable capillary hypersurfaces in a ball},
    journal={Math. Ann.},
    volume={374},
    pages={1845--1882},
    year={2019},
}

\bib{Wang-Xia2019}{article}{
	author={Wang, Guofang},
	author={Xia, Chao},
	title={Guan-Li type mean curvature flow for free boundary hypersurfaces in a ball},
	journal={Comm. Anal. Geom. (to appear)},
	eprint={arXiv:1910.07253},
}


\bib{Wei-X2022a}{article}{
	author={Wei, Yong},
	author={Xiong, Changwei},
	title={A volume-preserving anisotropic mean curvature type flow},
	journal={ Indiana Univ. Math. J.},
    volume={70},
	year={2021},
    number={3},
	pages={881--905},
}

\bib{Wei-X2022b}{article}{
	author={Wei, Yong},
	author={Xiong, Changwei},
	title={A fully nonlinear locally constrained anisotropic curvature flow},
	journal={ Nonlinear Anal.},
    volume={217},
	year={2022},
    pages={Paper No. 112760},
}

\bib{Weng-Xia2022}{article}{
	author={Weng, Liangjun},
	author={Xia, Chao},
	title={Alexandrov-Fenchel inequality for convex hypersurfaces with capillary boundary in a ball},
	journal={Trans. Amer. Math. Soc.},
	year={2022},
	volume={375},
    pages={8851--8883},
}


%
\bib{Whittlesey}{book}{
    author={Whittlesey, Marshall A.},
    title={Spherical Geometry and Its Applications},
    series={Textbooks in Mathematics},
    publisher={Chapman and Hall/CRC, New York},
    year={2019},
}

\end{thebibliography}
\end{document}